\documentclass[11pt]{article}

\usepackage[left=1in,top=1in,right=1in,bottom=1in,letterpaper]{geometry}

\usepackage[utf8]{inputenc} % allow utf-8 input
\usepackage[T1]{fontenc}    % use 8-bit T1 fonts
\usepackage{hyperref}       % hyperlinks
\usepackage{url}            % simple URL typesetting
\usepackage{booktabs}       % professional-quality tables
\usepackage{amsfonts}       % blackboard math symbols
\usepackage{nicefrac}       % compact symbols for 1/2, etc.
\usepackage{microtype}      % microtypography
\usepackage{xcolor}         % colors

\usepackage{color, colortbl}
\definecolor{LightCyan}{rgb}{0.88,1,1}
\usepackage{mathrsfs}
\usepackage{amssymb}
\title{\bf A Projection-free Algorithm for Constrained Stochastic Multi-level Composition Optimization}
\date{}

\author{Tesi Xiao\thanks{Department of Statistics, University of California, Davis. \texttt{texiao@ucdavis.edu}.  } 
\and Krishnakumar Balasubramanian\thanks{Department of Statistics, University of California, Davis \texttt{kbala@ucdavis.edu}. } 
\and Saeed Ghadimi\thanks{Department of Management Sciences, University of Waterloo. \texttt{sghadimi@uwaterloo.ca}.}
}

\usepackage{amsmath, amsthm, amsfonts}
\usepackage{algorithm, algorithmic}
\usepackage{enumitem, mathtools}

\newcommand{\proj}[0]{\Pi}

\DeclareMathOperator*{\argmin}{arg\,min}

\def\vx{x}
\def\vy{y}
\def\vz{z}

\def\vu{u}
\def\vv{v}
\def\vd{d}
\def\ve{e}
\def\vg{g}
\def\vf{f}
\def\vG{G}
\def\vJ{J}

\def\reals{\mathbb{R}} % Real number symbol
\def\setX{\mathcal{X}} % Set X
\def\setF{\mathscr{F}} % Set F
\renewcommand{\Pr}{\mathrm{Pr}}
\newcommand{\E}{\mathbb{E}}

\newtheorem{theorem}{Theorem}
\newtheorem{proposition}{Proposition}
\newtheorem*{remark}{Remark}
\newtheorem{lemma}[theorem]{Lemma}
\newtheorem{assumption}{Assumption}

\newtheorem{definition}[theorem]{Definition}

\begin{document}

\maketitle

\begin{abstract}%
We propose a projection-free conditional gradient-type algorithm for smooth stochastic multi-level composition optimization, where the objective function is a nested composition of $T$ functions and the constraint set is a closed convex set. Our algorithm assumes access to noisy evaluations of the functions and their gradients, through a stochastic first-order oracle satisfying certain standard unbiasedness and second-moment assumptions. We show that the number of calls to the stochastic first-order oracle and the linear-minimization oracle required by the proposed algorithm, to obtain an $\epsilon$-stationary solution, are of order $\mathcal{O}_T(\epsilon^{-2})$ and $\mathcal{O}_T(\epsilon^{-3})$ respectively, where $\mathcal{O}_T$ hides constants in $T$. Notably, the dependence of these complexity bounds on $\epsilon$ and $T$ are separate in the sense that changing one does not impact the dependence of the bounds on the other. For the case of $T=1$, we also provide a high-probability convergence result that depends poly-logarithmically on the inverse confidence level. Moreover, our algorithm is parameter-free and does not require any (increasing) order of mini-batches to converge unlike the common practice in the analysis of stochastic conditional gradient-type algorithms.
%, under standard assumptions (unbiased and bounded second moments) on the stochastic first-order oracle.
%in which the number of calls to the stochastic first-order oracle and the linear-minimization oracle, to obtain an $(\epsilon,\delta)$-stationary solution, are of order $\mathcal{O}(\epsilon^{-2}\log^2(1/\delta))$ and $\mathcal{O}(\epsilon^{-3}\log^3(1/\delta))$.
\end{abstract}

\section{Introduction}

We study projection-free algorithms for solving the following stochastic multi-level composition optimization problem
\begin{equation}\label{eq:opt-multi}
    \underset{\vx\in \setX}{\min} \quad F(\vx) := f_1 \circ \cdots \circ \vf_{T} (\vx),
\end{equation}
where $\vf_i: \reals^{d_{i}} \rightarrow \reals^{d_{i-1}}, i=1, \cdots, T ~(d_0=1)$ are continuously differentiable functions, the composite function $F$ is bounded below by $F^\star>-\infty$ and $\setX\subset\reals^d$ is a closed convex set. We assume that the exact function values and derivatives of $\vf_i$'s are not available. In particular, we assume that $\vf_i(\vy) = \E_{\xi_i}[G_i(\vy, \xi_i)]$ for some random variable $\xi_i$. Our goal is to solve the above optimization problem, given access to noisy  evaluations of $\nabla f_i$'s and $f_i$'s. 

There are two main challenges to address in developing efficient projection-free algorithms for solving~\eqref{eq:opt-multi}. First, note that denoting the transpose of the Jacobian of $f_i$ by $\nabla f_i$, the gradient of the objective function $F(x)$ in~\eqref{eq:opt-multi}, is given by $\nabla F(x) = \nabla f_T(y_T)\nabla f_{T-1}(y_{T-1}) \cdots \nabla f_1(y_1) $, where $y_i = f_{i+1} \circ \cdots \circ f_T(x)$ for $1 \leq i < T$, and $y_T = x$. Because of the nested nature of the gradient $\nabla F(x)$, obtaining an unbiased gradient estimator in the stochastic first-order setting, with controlled moments, becomes non-trivial. Using naive stochastic gradient estimators lead to oracle complexities that depend exponentially on $T$ (in terms of the accuracy parameter). Next, even when $T=1$ in the stochastic setting, projection-free algorithms like the conditional gradient method or its sliding variants invariably require increasing order of mini-batches\footnote{We discuss in detail about some recent works that avoid requiring increasing mini-batches, albeit under stronger assumptions, in Section~\ref{sec:relwork}.}~\cite{lan2016conditional,reddi2016stochastic, hazan2016variance,qu2018non, yurtsever2019conditional}, which make their practical implementation infeasible. %In the general $T$-level setting, using naive stochastic gradient estimator would lead to mini-batch order that depends exponentially on $T$.\todo{This seems to be true; can you double check}

In this work, we propose a projection-free conditional gradient-type algorithm that achieves \emph{level-independent} oracle complexities (i.e., the dependence of the complexities on the target accuracy is independent of $T$) using only \emph{one sample} of $(\xi_i)_{1\leq i \leq T}$ in each iteration, thereby addressing both of the above challenges. Our algorithm uses moving-average based stochastic estimators of the gradient and function values, also used recently by~\cite{ghadimi2020single} and \cite{balasubramanian2020stochastic}, along with an inexact conditional gradient method used by~\cite{balasubramanian2021zeroth} (which in turn is inspired by the sliding method by~\cite{lan2016conditional}). In order to establish our oracle complexity results, we use a novel merit function based convergence analysis. To the best of our knowledge, such an analysis technique is used for the first time in the context of analyzing stochastic conditional gradient-type algorithms.

\subsection{Preliminaries and Main Contributions}\label{sec:prelim}
We now introduce the technical preliminaries required to state and highlight the main contributions of this work. Throughout this work, $\|\cdot\|$ denotes the Euclidean norm for vectors and the Frobenius norm for matrices. We first describe the set of assumptions on the objective functions and the constraint set. 
\begin{assumption}[Constraint set]\label{aspt:constraint}
The set $\setX\subset \reals^d$ is convex and closed with $\underset{x,y \in \setX}{\max}~\|x-y\| \leq D_\setX$. % form some $D_\setX>0$.
\end{assumption}
\begin{assumption}[Smoothness]\label{aspt:lipschitz-T-level}
All functions $\vf_1, \dots, \vf_T$ and their derivatives are Lipschitz continuous with Lipschitz constants $L_{\vf_i}$ and $L_{\nabla \vf_i}$, respectively.
\end{assumption}
The above assumptions on the constraint set and the objective function are standard in the literature on stochastic optimization, and in particular in the analysis of conditional gradient algorithms and multi-level optimization; see, for example,~\cite{lan2016conditional},~\cite{yang2019multi-level} and ~\cite{balasubramanian2020stochastic}. We emphasize here that the above smoothness assumptions are made only on the functions $(f_i)_{1\leq i \leq T}$ and not on the stochastic functions $(G_i)_{1\leq i \leq T}$ (which would be a stronger assumption). Moreover, the Lipschitz continuity of $f_i$'s are implied by the Assumption~\ref{aspt:constraint} and assuming the functions are continuously differentiable. However, for sake of illustration, we state both assumptions separately. In addition to these assumptions, we also make unbiasedness and bounded-variance assumptions on the stochastic first-order oracle. Due to its technical nature, we state the precise details later in Section~\ref{sec:rates} (see Assumption~\ref{aspt:oracle-T-level}). 

We next turn our attention to the convergence criterion that we use in this work to evaluate the performance of the proposed algorithm. Gradient-based algorithms iteratively solve sub-problems in the form of
\begin{equation} \label{main_subproblem}
\underset{u \in \setX}{\argmin} \left\{ \langle g,u \rangle + \frac{\beta}{2} \| u -x\|^2 \right\},
\end{equation}
for some $\beta>0$, $g\in\mathbb{R}^d$ and $x \in \setX$. 
Denoting the optimal solution of the above problem by $P_\setX (x,g,\beta)$ and noting its optimality condition, one can easily show that
\[
-\nabla F(\bar{\vx}) \in \mathcal{N}_{\setX}(\bar{\vx})+{\cal B}\Big(0,\|g-\nabla F(\bar{\vx})\|D_{\setX}+\beta\|x-P_\setX (x,g,\beta)\|(D_{\setX}+\|\nabla F(\bar{\vx})\|/\beta)\Big),
\]
where $\mathcal{N}_{\setX}(\bar{\vx})$ is the normal cone to $\setX$ at $\bar{\vx}$ and ${\cal B}(0,r)$ denotes a ball centered at the origin with radius $r$. Thus, reducing the radius of the ball in the above relation will result in finding an approximate first-order stationary point of the problem for non-convex constrained minimization problems.
%in the form of \eqref{eq:opt-multi}, a first-order stationary point $\bar{\vx}$ of the problem \eqref{eq:opt-multi} should satisfy $$-\nabla F(\bar{\vx}) \in \mathcal{N}_{\setX}(\bar{\vx}), $$  
%where $\mathcal{N}_{\setX}(\bar{\vx})$ is the normal cone to $\setX$ at $\bar{\vx}$. \todo{Add more motivation here} 
Motivated by this fact, we can define the gradient mapping at point $\bar x \in \setX$ as 
\begin{equation}\label{eq:definition-gradient-mapping}
    \mathcal{G}_{\setX}(\bar{\vx}, \nabla F(\bar{\vx}), \beta) \coloneqq \beta \left(\bar{\vx} - P_\setX (\bar x,\nabla F(\bar{\vx}),\beta)\right)=
    \beta \left(\bar{\vx} - \proj_{\setX}\left(\bar{\vx}- \frac{1}{\beta}\nabla F(\bar{\vx})\right)\right),
\end{equation}
where $\proj_{\setX}(y)$ denotes the Euclidean projection of the vector $y$ onto the set $\setX$.  The gradient mapping is a classical measure has been widely used in the literature as a convergence criterion when solving nonconvex constrained problems \cite{nesterov2018lectures}. It plays an analogous role of the gradient for constrained optimization problems; in fact when the set $\setX \equiv \mathbb{R}^d$ the gradient mapping just reduces to $\nabla F(\bar{\vx})$. It should be emphasized that while the gradient mapping cannot be computed in the stochastic setting, it still serves as a measure of convergence. Our main goal in this work is to find an $\epsilon$-stationary solution to~\eqref{eq:opt-multi}, in the sense described below. % want to find the solution $\bar{\vx}$ such that $\| \mathcal{G}_\setX(\bar{\vx}, \nabla F(\bar{\vx}), \beta)\|^2 \leq \epsilon$.

\begin{definition}\label{def:gradmapping}
A point $\bar{\vx}\in\setX$ generated by an algorithm for solving~\eqref{eq:opt-multi} is called an $\epsilon$-stationary point, if we have $\E[\| \mathcal{G}_{\setX}(\bar{\vx}, \nabla F(\bar{\vx}), \beta)\|^2]\leq \epsilon$, where the expectation is taken over all the randomness involved in the problem.
\end{definition}
In the literature on conditional gradient methods for the nonconvex setting, the so-called Frank-Wolfe gap is also widely used to provide convergence analysis. The Frank-Wolfe Gap is defined formally as
\begin{equation}\label{eq:definition-fw-gap}
    g_{\setX}(\bar{\vx}, \nabla F(\bar{\vx})) := \underset{\vy\in\setX}{\max}~\langle \nabla F(\bar{\vx}) , \bar{\vx} - \vy  \rangle.
\end{equation}
As pointed out by~\cite{balasubramanian2021zeroth}, the gradient mapping criterion and the Frank-Wolfe gap are related to each other in the following sense.

\begin{proposition}{\normalfont \cite{balasubramanian2021zeroth}}
Let $g_{\setX}(\cdot)$ be the Frank-Wolfe gap defined in \eqref{eq:definition-fw-gap} and $\mathcal{G}_{\setX}(\cdot)$ be the gradient mapping defined in \eqref{eq:definition-gradient-mapping}. Then, we have $\| \mathcal{G}_{\setX}(\vx, \nabla F(\vx), \beta)\|^2 \leq g_{\setX}(\vx, \nabla F(\vx)), \forall \vx\in\setX$.
Moreover, under Assumption \ref{aspt:constraint}, \ref{aspt:lipschitz-T-level}, we have $g_{\setX}(\vx, \nabla F(\vx)) \leq \left[(1/\beta)\prod_{i=1}^T L_{\vf_i}+D_\setX\right]\| \mathcal{G}_{\setX}(\vx, \nabla F(\vx), \beta)\|$.
\end{proposition}
For stochastic conditional gradient-type algorithms, the oracle complexity is measured in terms of number of calls to the Stochastic First-order Oracle (SFO) and the Linear Minimization Oracle (LMO) used to the solve the sub-problems (that are of the form of minimizing a linear function over the convex feasible set) arising in the algorithm. In this work, we hence measure the number of calls to SFO and LMO required by the proposed algorithm to obtain an $\epsilon$-stationary solution in the sense of Definition~\ref{def:gradmapping}. We now highlight our \textbf{main contributions}:

\begin{itemize}[leftmargin=0.2in]
    \item We propose a projection-free conditional gradient-type method (Algorithm~\ref{alg:CG-NASA-T-level}) for solving~\eqref{eq:opt-multi}. In Theorem~\ref{thm:main-T-level}, we show that the SFO and LMO complexities of this algorithm, in order to achieve an $\epsilon$-stationary solution in the sense of Definition~\ref{def:gradmapping}, are of order $\mathcal{O}(\epsilon^{-2})$ and $\mathcal{O}(\epsilon^{-3})$, respectively.
    \item The above SFO and LMO complexities are in particular level-independent (i.e., the dependence of the complexities on the target accuracy is independent of $T$). The proposed algorithm is parameter-free and does not require any mini-batches, making it applicable for the online setting. \item When considering the case of $T \leq 2$, we present a simplified method (Algorithm~\ref{alg:CG-NASA} and~\ref{alg:CG-ASA}) to obtain the same oracle complexities. Intriguingly, while this simplified method is still parameter-free for $T=1$, it is not when $T=2$ (see Theorem~\ref{thm:two-level} and Remark~\ref{rmk:parameterfree}). Furthermore, for the case of $T=1$, we also establish high-probability bounds (see Theorem~\ref{thm:one-level-highprob}).  
\end{itemize}
A summary of oracle complexities for stochastic conditional gradient-type algorithms is in Table~\ref{tab:summary}. 
\begin{table}[t]
\small
\centering
\resizebox{\columnwidth}{!}{%
\begin{tabular}{l c c c c  c}
\toprule
 \textbf{Algorithm}  & \textbf{Criterion} & \textbf{\# of levels} & \textbf{Batch size} & \textbf{SFO} & \textbf{LMO} \\
 \midrule
  SPIFER-SFW \cite{yurtsever2019conditional} & FW-gap & 1 & $\mathcal{O}(\epsilon^{-1})$ & $\mathcal{O}(\epsilon^{-3})$ & $\mathcal{O}(\epsilon^{-2})$\\
 \hline
 1-SFW \cite{zhang2020one} & FW-gap & 1 & 1  & $\mathcal{O}(\epsilon^{-3})$ & $\mathcal{O}(\epsilon^{-3})$\\
  \hline
 SCFW \cite{akhtar2021projection} & FW-gap & 2 & 1  & $\mathcal{O}(\epsilon^{-3})$ & $\mathcal{O}(\epsilon^{-3})$\\
 \hline
 SCGS \cite{qu2018non} & GM & 1 & $\mathcal{O}(\epsilon^{-1})$ &  $\mathcal{O}(\epsilon^{-2})$&  $\mathcal{O}(\epsilon^{-2})$\\
 \hline
  SGD+ICG \cite{balasubramanian2021zeroth} & GM & 1 & $\mathcal{O}(\epsilon^{-1})$ &  $\mathcal{O}(\epsilon^{-2})$&  $\mathcal{O}(\epsilon^{-2})$\\
 \hline
\rowcolor{LightCyan}
  LiNASA+ICG (Algorithm~\ref{alg:CG-NASA-T-level}) & GM & $T$   & 1 &  $\mathcal{O}_T(\epsilon^{-2})$&  $\mathcal{O}_T(\epsilon^{-3})$\\
\bottomrule
\end{tabular}
}
\caption{Complexity results for stochastic conditional gradient type algorithms to find an $\epsilon$-stationary solution in the nonconvex setting. FW-Gap and GM stands for Frank-Wolfe Gap (see~\eqref{eq:definition-fw-gap}) and Gradient Mapping (see~\eqref{eq:definition-gradient-mapping}) respectively. $\mathcal{O}_T$ hides constants in $T$. Existing one-sample based stochastic conditional gradient algorithms are either (i) not applicable to the case of general $T> 1$, or (ii) require strong assumptions~\cite{zhang2020one}, or (iii) are not truly online~\cite{akhtar2021projection}; see Section~\ref{sec:relwork} for detailed discussion. The results in~\cite{balasubramanian2021zeroth} are actually presented for the zeroth-order setting; however the above stated first-order complexities follow immediately.}
\label{tab:summary}
\end{table}

\subsection{Related Work}\label{sec:relwork}

\noindent\textbf{Conditional Gradient-Type Method.} The conditional gradient algorithm \cite{frank1956algorithm, levitin1966constrained}, has had a renewed interest in the machine learning and optimization communities in the past decade; see~\cite{migdalas1994regularization,jaggi2013revisiting,harchaoui2015conditional, lacoste2015global, beck2017linearly, garber2021improved} for a partial list of recent works.  Considering the stochastic convex setup,~\cite{hazan2012projection, hazan2016variance} provided expected oracle complexity results for the stochastic conditional gradient algorithm. The complexities were further improved by a sliding procedure in~\cite{lan2016conditional}, based on Nesterov's acceleration method.~\cite{reddi2016stochastic, yurtsever2019conditional, hazan2016variance} considered variance reduced stochastic conditional gradient algorithms, and provided expected oracle complexities in the non-convex setting. \cite{qu2018non} analyzed the sliding algorithm in the non-convex setting and provided results for the gradient mapping criterion. \emph{All of the above works require increasing orders of mini-batches to obtain their oracle complexity results}.

\cite{mokhtari2020stochastic} and~\cite{zhang2020one} proposed a stochastic conditional gradient-type algorithm with moving-average gradient estimator for the convex and non-convex setting that uses only one-sample in each iteration. However,~\cite{mokhtari2020stochastic} and~\cite{zhang2020one} require several restrictive assumptions, which we explain next (focusing on~\cite{zhang2020one} which considers the nonconvex case). Specifically,~\cite{zhang2020one} requires that the stochastic function $G_1(x,\xi_1)$ has uniformly bounded function value, gradient-norm, and Hessian spectral-norm, and the distribution of the random vector $\xi_1$ has an absolutely continuous density $p$ such that the norm of the gradient of $\log p$ and spectral norm of the Hessian of $\log p$ has finite fourth and second-moments respectively. In contrasts, we do not require such stringent assumptions. Next, all of the above works focus only on the case of $T=1$. ~\cite{akhtar2021projection} considered stochastic conditional gradient algorithm for solving~\eqref{eq:opt-multi}, with $T=2$. However,~\cite{akhtar2021projection} also makes stringent assumptions: (i) the stochastic objective functions $G_1(x,\xi_1)$  and $G_2(x,\xi_1)$ themselves have Lipschitz gradients almost surely and (ii) for a given instance of random vectors $\xi_1$ and $\xi_2$, one could query the oracle at the current and previous iterations, which makes the algorithm not to be truly online. See Table~\ref{tab:summary} for a summary. \\

\noindent\textbf{Stochastic Multi-level Composition Optimization.} Compositional optimization problems of the form in~\eqref{eq:opt-multi} have been considered as early as 1970s by~\cite{Ermolievold}. Recently, there has been a renewed interest on this problem.~\cite{ermoliev2013sample} and~\cite{dentcheva2017statistical} considered a sample-average approximation approach for solving~\eqref{eq:opt-multi} and established several asymptotic results. For the case of $T=2$,~\cite{wang2017stochastic}, \cite{wang2016accelerating} and~\cite{blanchet2017unbiased} proposed and analyzed stochastic gradient descent-type algorithms in the smooth setting.~\cite{davis2019stochastic} and \cite{duchi2018stochastic} considered the non-smooth setting and established oracle complexity results. Furthermore,~\cite{hu2020biased} proposed algorithms when the randomness between the two levels are not necessarily independent. For the general case of $T\geq 1$, \cite{yang2019multi-level} proposed stochastic gradient descent-type algorithms and established oracle complexities established that depend exponentially on $T$ and are hence sub-optimal. Indeed, large deviation and Central Limit Theorem results established in~\cite{ermoliev2013sample, dentcheva2017statistical}, respectively, show that in the sample-average approximation setting, the $\argmin$ of the problem in~\eqref{eq:opt-multi} based on $n$ samples, converges at a level-independent rate (i.e., dependence of the convergence rate on the target accuracy is independent of $T$) to the true minimizer, under suitable regularity conditions. 

\cite{ghadimi2020single} proposed a single time-scale Nested Averaged Stochastic Approximation (NASA) algorithm and established optimal rates for the cases of $T=1,2$. For the general case of $T\geq 1$,~\cite{balasubramanian2020stochastic} proposed a linearized NASA algorithm and established level-independent and optimal convergence rates. Concurrently,~\cite{ruszczynski2021stochastic} considered the case when the function $f_i$ are non-smooth and established asymptotic convergence results. \cite{zhang2019multi} also established non-asymptotic level-independent oracle complexities, however, under stronger assumptions than that in~\cite{balasubramanian2020stochastic}. Firstly, they assumed that for a fixed batch of samples, one could query the oracle on different points, which is not suited for the general online stochastic optimization setup. Next, they assume a much stronger mean-square Lipschitz smoothness assumption on the individual functions $f_i$ and their gradients. Finally, they required mini-batches sizes that depend exponentially on $T$, which makes their method impractical. Concurrent to~\cite{balasubramanian2020stochastic}, level-independent rates were also obtained for \emph{unconstrained} problems by~\cite{chen2021solving}, albeit, under the stronger assumption that the stochastic functions $G_i(x,\xi_i)$ are Lipschitz, almost surely. It is also worth mentioning that while some of the above papers considered constrained problems, the algorithms proposed and analyzed in the above works are not projection-free, which is the main focus of this work. \vspace{-0.05in}

%\noindent{\textbf{Organization.}} The rest of the paper is organized as follows. In Section~\ref{sec:method}, we introduce the \texttt{LiNASA+ICG} algorithm, along with its motivation. In Section~\ref{sec:rates}, we present our main theoretical result on the SFO and LMO complexity of the \texttt{LiNASA+ICG} algorithm. In Section~\ref{sec:specialcase}, we consider the special cases of $T=1, 2$ and discuss several subtle issues. We provide a proof-sketch in Section~\ref{sec:sketch}. The full proofs are presented in Appendix -- Sections~\ref{sec:tech-lemma},~\ref{sec:proof-thm-T} , ~\ref{sec:proof-thm-one-two}, and~\ref{sec:proof-high-prob}. 

%The authors showed that by modifying the specific Lyapunov function, defined in \cite{ruszczynski1987linearization} for nonsmooth single-level stochastic optimization, the convergence analysis of the NASA algorithm can be established such that its complexity bound matches the case of $T=1$. This resolved the above question for $T=2$. However, constructing similar algorithms for the case of general $T$ had remained less investigated.

\section{Methodology}\label{sec:method}%\vspace{-0.1in}

In this section, we present our projection-free algorithm for solving problem \eqref{eq:opt-multi}. The method generates three random sequences, namely, approximate solutions $\{\vx^k\}$, average gradients $\{\vz^k\}$, and average function values $\{\vu^k\}$, defined on a certain probability space $(\Omega,\setF,P)$. We let $\setF_k$ to be the $\sigma$-algebra generated by $\{\vx^0, \dots, \vx^k, \vz^0, \dots, \vz^k, \vu_1^0, \dots, \vu_1^k, \dots ,\vu_T^0, \dots, \vu_T^k\}$. The overall method is given in Algorithm \ref{alg:CG-NASA-T-level}. In~\eqref{eq:update-z-T-level}, the stochastic  Jacobians $\vJ_i^{k+1}\in \reals^{d_{i}\times d_{i-1}}$, and the product $\prod_{i=1}^{T} \vJ_{T+1-i}^{k+1}$ is calculated as $\vJ_{T}^{k+1}\vJ_{T-1}^{k+1}\cdots\vJ_{1}^{k+1}\in \reals^{d_T\times d_1} \equiv\reals^{d_T\times 1}$. In~\eqref{eq:update-u-T-level}, we use the notation $\langle\cdot, \cdot \rangle$ to represent both matrix-vector multiplication and vector-vector inner product. There are two aspects of the algorithm that we highlight specifically: (\texttt{i})  In
addition to estimating the gradient of $F$, we also estimate a  stochastic linear approximation of the inner functions $f_i$ by a moving-average technique. In the multi-level setting we consider, it helps us to avoid the accumulation of bias, when estimating the $f_i$ directly. Linearization techniques were used in the stochastic optimization since the work of~\cite{ruszczynski1987linearization}. A similar approach was used in~\cite{balasubramanian2020stochastic} in the context of projected-based methods for solving~\eqref{eq:opt-multi}. It is also worth mentioning that other linearization techniques have been used in~\cite{davis2019stochastic} and~\cite{duchi2018stochastic} for estimating the stochastic inner function values for weakly convex two-level composition problems, and (\texttt{ii}) The \texttt{ICG} method given in Algorithm \ref{alg:ICG} is essentially applying \emph{deterministic} conditional gradient method with the exact line search for solving the quadratic minimization subproblem in \eqref{main_subproblem} with the estimated gradient $z_k$ in~\eqref{eq:update-z-T-level}. It was also used in~\cite{balasubramanian2021zeroth} as a sub-routine and is motivated by the sliding approach of~\cite{lan2016conditional}.

\begin{algorithm}[t]
\caption{Linearized NASA with Inexact Conditional Gradient Method (\texttt{LiNASA+ICG})}
\begin{algorithmic}
    	\STATE \textbf{Input:} $\vx^0\in \setX$, $\vz^0=0\in\reals^d$, $\vu_i^0\in\reals^{d_i} , i=1,\dots, T$, $\beta_k>0$, $t_k>0$, $\tau_k \in (0, 1]$, $\delta \geq 0$.
        \FOR{$k =  0, 1 , 2, \dots, N$}
        \STATE 1. Update the solution:
        \begin{align}\label{eq:subproblem-x}
        \tilde{\vy}^{k} &= \texttt{ICG}(\vx^k, \vz^k, \beta_k, t_k, \delta),\\
        \vx^{k+1} &= \vx^k + \tau_k (\tilde{\vy}^{k}-\vx^k),\label{eq:update-x-T-level}
        \end{align}
        and compute stochastic Jacobians $\vJ_i^{k+1}$, and function values $\vG_{i}^{k+1}$ at $\vu_{i+1}^k$ for $i=1,\dots, T$.
        \STATE 2. Update average gradients $\vz$ and function value estimates $\vu_{i}$ for each level $i=1,\dots, T$
        \begin{align}
        \vz^{k+1} &= (1- \tau_k)\vz^k + \tau_k \prod_{i=1}^{T} \vJ_{T+1-i}^{k+1},\label{eq:update-z-T-level}\\
        \vu_{i}^{k+1} &= (1-\tau_k)\vu_{i}^k + \tau_k \vG_{i}^{k+1} + \langle \vJ_{i}^{k+1}, \vu_{i+1}^{k+1} - \vu_{i+1}^k\rangle.\label{eq:update-u-T-level}
        \end{align}
        \ENDFOR
        \STATE \textbf{Output:} $(\vx^{R}, \vz^R, \vu_1^R, \cdots, \vu_T^R)$, where $R$ is uniformly distributed over $\{1,2,\dots, N\}$
\end{algorithmic}
\label{alg:CG-NASA-T-level}
\end{algorithm}

\begin{algorithm}[t]
\caption{Inexact Conditional Gradient Method (\texttt{ICG})}
\begin{algorithmic}
	  \STATE \textbf{Input:} ($\vx, \vz, \beta, M, \delta)$\\
	  \textbf{Set} $w^0=\vx$.
      \FOR{$t =  0, 1 , 2, \dots, M$} 
         \STATE 1. Find $\vv^t\in\setX$ with a quantity $\delta \geq 0$ such that
          \begin{equation*}
              \langle \vz+\beta(w^t- \vx), \vv^t\rangle \leq \min_{\vv\in\setX} ~ \langle \vz+\beta(w^t- \vx), \vv\rangle + \frac{ \beta D_\setX^2 \delta}{t+2}. 
          \end{equation*}
          \STATE 2. Set $w^{t+1} = (1-\mu_t) w^t + \mu_t \vv^t$ with $\mu_t = \min\bigg\{ 1, \frac{\langle\beta(\vx-w^t) -\vz, \vv^t-w^t \rangle}{\beta\|\vv^t - w^t\|^2} \bigg\}.$
      \ENDFOR
      \STATE \textbf{Output:} $w^{M}$
\end{algorithmic}
\label{alg:ICG}
\end{algorithm}

\section{Main Results}\label{sec:rates}
In this section, we present our main result on the oracle complexity of Algorithm~\ref{alg:CG-NASA-T-level}. Before we proceed, we present our assumptions on the stochastic first-order oracle.

\begin{assumption}[Stochastic First-Order Oracle]\label{aspt:oracle-T-level}
Denote $\vu_{T+1}^k \equiv \vx^k$. For each $k$, $\vu_{i+1}^k$ being the input, the stochastic oracle outputs $\vG_i^{k+1}\in\reals^{d_i}$ and $\vJ_{i}^{k+1}$ such that given $\setF_k$ and for any $i\in\{1,\dots,T\}$
\begin{itemize}[leftmargin=2em]
    \item[(a)] $\E[\vJ_{i}^{k+1}|\setF_k] = \nabla \vf_i(\vu _{i+1}^k),~\E[\vG_{i}^{k+1}|\setF_k] = \vf_i(\vu_{i+1}^k)$,
    \item[(b)] $\E[\|\vG_i^{k+1} - \vf_i(\vu_{i+1}^k) \|^2|\setF_k]\leq \sigma_{\vG_i}^2, ~\E[\|\vJ_i^{k+1} - \nabla \vf_i(\vu_{i+1}^k) \|^2|\setF_k]\leq \sigma_{\vJ_i}^2$,
    \item[(c)] The outputs of the stochastic oracle at level $i$, $\vG_i^{k+1}$ and $\vJ_{i}^{k+1}$, are independent. The outputs of the stochastic oracle are independent between levels, i.e., $\{\vG_{i}^{k+1}\}_{i=1,\dots,T}$ are independent and so are $\{\vJ_{i}^{k+1}\}_{i=1,\dots,T}$.
\end{itemize}
\end{assumption}
Parts (a) and (b) in Assumption~\ref{aspt:oracle-T-level} are standard unbiasedness and bounded variance assumptions on the stochastic gradient, common in the literature. Part (c) is essential to establish the convergence results in the multi-level case. Similar assumptions have been made, for example, in~\cite{yang2019multi-level} and~\cite{balasubramanian2020stochastic}. We also emphasize that unlike some prior works (see e.g.,~\cite{zhang2020one}), Assumption~\ref{aspt:oracle-T-level} allows the case of endogenous uncertainty, and we do not require the distribution of the random variables $(\xi_i)_{1\leq i \leq T}$ to be independent of the distribution of the decision variables $(u_i)_{1\leq i \leq T}$.

\begin{remark} Under Assumption \ref{aspt:lipschitz-T-level}, and \ref{aspt:oracle-T-level}, we can immediately conclude that $\E[\| \vJ_{i}^{k+1} \|^2|\setF_k] 
    = \E[\| \vJ_{i}^{k+1}  - \nabla \vf_{i}(\vu_{i+1}^k)\|^2|\setF_k] + \| \nabla \vf_{i}(\vu_{i+1}^k) \|^2\leq \sigma_{\vJ_i}^2 + L_{\vf_i}^2 := \hat{\sigma}_{\vJ_i}^2$.
In the sequel, $\hat{\sigma}_{\vJ_i}^2$ will be used to simplify the presentation. 
\end{remark}

We start with the merit function used in this work and its connection to the gradient mapping criterion. Our proof leverages the following merit function:
\begin{equation}\label{eq:definition-merit-fucntion-T-level}
    W_{\alpha,\gamma}(\vx, \vz, \vu) = F(\vx) -F^\star - \eta (\vx, \vz) + \alpha\| \nabla F(\vx) - \vz \|^2 + \sum_{i=1}^{T}\gamma_i \| \vf_i(\vu_{i+1}) - \vu_{i}\|^2,
\end{equation}
where $\alpha, \{\gamma_i\}_{1\leq i\leq T}$ are positive constants and 
\begin{equation}\label{eq:definition-eta}
    \eta(\vx, \vz) =\min _{\vy \in \setX}\left\{  H(\vy; \vx, \vz, \beta):=  \langle \vz, \vy-\vx\rangle+\frac{\beta}{2}\|\vy-\vx\|^{2}\right\}.
\end{equation}

Compared to \cite{balasubramanian2020stochastic}, we require the additional term  $\| \nabla F(\vx) - \vz \|^2$, which turns out to be essential in our proof due to the \texttt{ICG} routine. The following proposition relates the merit function above to the gradient mapping.
\begin{proposition}\label{prop:merit-function-grad-mapping}
Let $\mathcal{G}_{\setX}(\cdot)$ be the gradient mapping defined in \eqref{eq:definition-gradient-mapping} and $\eta(\cdot,\cdot)$ be defined in \eqref{eq:definition-eta}. For any pair of $(\vx,\vz)$ and $\beta>0$, we have $\|\mathcal{G}_{\setX}(\vx, \nabla F(\vx), \beta)\|^2 \leq -4\beta  \eta(\vx, \vz) + 2\|\nabla F(\vx)-\vz\|^2$.
\end{proposition}
\begin{proof}
By expanding the square, and using the properties of projection operation, we have $$\|\proj_{\setX}({\vx} - \frac{1}{\beta}{\vz})  - {\vx}\|^2 + \|\proj_{\setX}({\vx} - \frac{1}{\beta}{\vz}) - (\vx-\frac{1}{\beta}{\vz}) \|^2 \leq \| \bar{\vx} - ({\vx}-\frac{1}{\beta}{\vz})\|^2 = \|\frac{1}{\beta}{\vz}\|^2.$$ 
Thus, we have $\eta({\vx}, {\vz}) \leq -\frac{\beta}{2}\|\proj_{\setX}({\vx}-\frac{1}{\beta}{\vz}) -{\vx}\|^2$. The proof is completed immediately by noting that $ \| \mathcal{G}({\vx}, \nabla F({\vx}), \beta) \|^2 \leq  2\beta^2 \| \proj_{\setX}({\vx}-\frac{1}{\beta}{\vz}) - {\vx}\|^2 +2\left\|\nabla F({\vx}) -{\vz}\right\|^2.$
\end{proof}

We now present out main result on the oracle complexity of Algorithm\ref{alg:CG-NASA-T-level}. 
\begin{theorem}\label{thm:main-T-level}
Under Assumption \ref{aspt:constraint}, \ref{aspt:lipschitz-T-level}, \ref{aspt:oracle-T-level}, let $\{\vx^k, \vz^k, \{\vu_i^k\}_{1\leq i\leq T}\}_{k\geq 0}$ be the sequence generated by Algorithm \ref{alg:CG-NASA-T-level} with $N\geq 1$ and
\begin{equation}\label{thm:parameter-value-T-level}
\begin{split}
    \beta_k \equiv \beta >0, \qquad \tau_0 = 1,~ t_0 = 0,\quad \tau_k = \frac{1}{\sqrt{N}},~t_k = \lceil\sqrt{k}\rceil, \quad \forall k\geq1, 
\end{split}
\end{equation}
where $\beta$ is an arbitrary positive constant. Provided that the merit function $W_{\alpha,\gamma}(\vx,\vz, \vu)$ is defined as \eqref{eq:definition-merit-fucntion-T-level} with
\begin{equation}\label{eq:alpha-gamma-values}
    \alpha = \frac{\beta}{20 L_{\nabla F}^2},\quad \gamma_1 = \frac{\beta}{2} ,\quad \gamma_j =  \left(2\alpha + \frac{1}{4\alpha L_{\nabla F}^2}\right) (T-1)C_j^2 + \frac{\beta}{2},  \quad 2\leq j\leq T,
\end{equation}
we have,
\begin{equation}\label{eq:thm-gradient-mapping}
    \E \left[ \| \mathcal{G}_\setX(\vx^R, \nabla F(\vx^R), \beta) \|^2\right] \leq \frac{2 (\beta + \frac{20L_{\nabla F}^2}{\beta})\left[2W_{\alpha,\gamma}(\vx^0, \vz^0, \vu^0) + \mathcal{B}(\beta, \sigma^2, L, D_{\setX}, T, \delta)\right]}{\sqrt{N}},
\end{equation}
\begin{equation}\label{eq:thm-function-value-error}
    \E \left[ \| \vf_i(\vu_{i+1}^R) - \vu_{i}^R \|^2 \right]\leq \frac{2W_{\alpha,\gamma}(\vx^0, \vz^0, \vu^0) + \mathcal{B}(\beta, \sigma^2, L, D_{\setX}, T, \delta)}{\beta\sqrt{N}}, \quad 1\leq i\leq T.
\end{equation}
where $\vu_{T+1} = \vx, \mathcal{B}(\beta, \sigma^2, L, D_{\setX}, T, \delta) = 4\hat{\sigma}^2 + 32\beta D_{\setX}^2(1+\delta) \left(\frac{3}{5} + \frac{5L_{\nabla F}^2}{\beta^2}\right)$, and $\hat{\sigma}^2$ is a constant depending on the parameters $(\beta, \sigma^2, L, D_{\setX}, T)$ given in \eqref{eq:definition-sigma-hat}.
The expectation is taken with respect to all random sequences generated by the method and an independent random integer number $R$ uniformly distributed over $\{1,\dots,N\}$. That is to say, the number of calls to SFO and LMO to get an $\epsilon$-stationary point is upper bounded by $\mathcal{O}_{T}(\epsilon^{-2}), \mathcal{O}_{T}(\epsilon^{-3})$ respectively.
\end{theorem}

\begin{remark}
The constant $\mathcal{B}(\beta, \sigma^2, L, D_\setX, T, \delta)$ is $\mathcal{O}(T)$ given the definition of $\hat{\sigma}^2$ and the value of $\gamma_j$ in \eqref{eq:alpha-gamma-values}, which further implies that the total number of calls to SFO and LMO of Algorithm \ref{alg:CG-NASA-T-level} for finding an $\epsilon$-stationary point of \eqref{eq:opt-multi}, are bounded by $\mathcal{O}(T^2\epsilon^{-2}) = \mathcal{O}_T(\epsilon^{-2})$ and $\mathcal{O}(T^3\epsilon^{-3}) = \mathcal{O}_T(\epsilon^{-3})$ respectively. Furthermore, it is worth noting that this complexity bound for Algorithm \ref{alg:CG-NASA-T-level} is obtained without any dependence of the parameter $\beta_k$ on Lipschitz constants due to the choice of arbitrary positive constant $\beta$ in \eqref{thm:parameter-value-T-level}, and $\tau_k, t_k$ depend only on the number of iterations $N$ and $k$ respectively. This makes Algorithm \ref{alg:CG-NASA-T-level} parameter-free and easy to implement.
\end{remark}

\begin{remark}
As discussed in Section~\ref{sec:method}, the \texttt{ICG} routine given in Algorithm \ref{alg:ICG} is a \emph{deterministic} method with the estimated gradient $z_k$ in~\eqref{eq:update-z-T-level}. The number of iterations, $t_k$, required to run Algorithm \ref{alg:ICG} is given by $t_k = \lceil\sqrt{k}\rceil$. That is, we require more precise solutions for the \texttt{ICG} routine, only for later outer iterations. Furthermore, due to the deterministic nature of the \texttt{ICG} routine, further advances in the analysis of deterministic conditional gradient methods under additional assumptions on the constraint set $\setX$  (see, for example, \cite{garber2015faster, garber2021frank}) could be leveraged to improve the overall LMO complexity. 
\end{remark}

\subsection{The special cases of $T=1$ and $T=2$} \label{sec:specialcase}
We now discuss several intriguing points regarding the choice of tuning parameter $\beta$, for  the case of $T=2$, and the more standard case of $T=1$. Specifically, the linearization technique used in Algorithm~\ref{alg:CG-NASA-T-level} turns out to be not necessary for the case of $T=2$ and $T=1$ to obtain similar rates. However, without linearization, the choice of $\beta$ is dependent on the problem parameters for $T=2$. Whereas it turns out to be independent of the problem parameters (similar to Algorithm~\ref{alg:CG-NASA-T-level} and Theorem~\ref{thm:main-T-level} which holds for all $T\geq 1$) for $T=1$. As the outer function value estimates (i.e., $u^{k+1}_1$ sequence) are not required for the convergence analysis, we remove them in Algorithms~\ref{alg:CG-NASA} and~\ref{alg:CG-ASA}.

\begin{algorithm}[H]
\caption{NASA with Inexact Conditional Gradient Method (\texttt{NASA+ICG}) for $T=2$}
\begin{algorithmic}
        \STATE Replace Step 2 of Algorithm \ref{alg:CG-NASA-T-level} with the following:
        \STATE 2'. Update the average gradient $\vz$ and the function value estimate $\vu_2$ respectively as:
        \begin{align*}
          \vz^{k+1} = (1- \tau_k)\vz^k + \tau_k \vJ_2^{k+1} \vJ_1^{k+1}~\quad\text{and}\quad
          \vu_2^{k+1} = (1-\tau_k)\vu^k + \tau_k \vG_2^{k+1}
        \end{align*}
\end{algorithmic}
\label{alg:CG-NASA}
\end{algorithm}

\begin{algorithm}[H]
\caption{ASA with Inexact Conditional Gradient Method (\texttt{ASA+ICG}) for $T=1$}
\begin{algorithmic}
        \STATE Replace Step 2 of Algorithm \ref{alg:CG-NASA-T-level} with the following:
        \STATE 2''. Update the average gradient $\vz$ as: $ \vz^{k+1} = (1- \tau_k)\vz^k + \tau_k \vJ_1^{k+1}.$
        
\end{algorithmic}
\label{alg:CG-ASA}
\end{algorithm}

\begin{theorem}\label{thm:two-level}
Let Assumptions \ref{aspt:constraint}, \ref{aspt:lipschitz-T-level}, \ref{aspt:oracle-T-level} be satisfied by the optimization problem~\eqref{eq:opt-multi}. Let $\mathcal{C}_1, \mathcal{C}_2$ and $\mathcal{C}_3$ be some constants depending on the parameters $(\beta, \sigma^2, L, D_{\setX}, \delta)$, as defined in \eqref{eq:definition-C1-C2} and \eqref{eq:definition-C3}. Let $\tau_0 = 1, t_0 = 0$, $\tau_k = \frac{1}{\sqrt{N}}, t_k = \lceil\sqrt{k}\rceil, \forall k\geq1$, where $N$ is the total number of iterations. 
\begin{itemize}[leftmargin=2em]
    \item[(a)] Let $T=2$, and let $\{\vx^k, \vz^k, \vu_2^k\}_{k\geq 0}$ be the sequence generated by Algorithm \ref{alg:CG-NASA} with
\begin{equation}\label{thm:parameter-value-2-level}
    \beta_k \equiv\beta \geq 6\rho L_{\nabla F} + (2\rho +\frac{2}{3\rho})L_{\nabla f_1}L_{\vf_2}^2, \quad \rho>0.
\end{equation}
Then, we have $\forall N \geq 1$,
\begin{equation*}
        \E \left[ \| \mathcal{G}_{\setX}(\vx^R, \nabla F(\vx^R), \beta) \|^2\right] \leq \frac{\mathcal{C}_1}{\sqrt{N}}, \quad \E \left[ \| \vf_2(\vx^R) - \vu_2^R \|^2 \right]\leq \frac{\mathcal{C}_2}{\sqrt{N}}.
\end{equation*}
%where $\mathcal{C}_1, \mathcal{C}_1$ are constants depending on the parameters $(\beta, \sigma^2, L, D_{\setX}, \delta)$.
\item[(b)] Let $T=1$ and let $\{\vx^k, \vz^k\}_{k\geq 0}$ be the sequence generated by Algorithm \ref{alg:CG-ASA} with $\beta_k \equiv\beta>0$. Then, we have $\forall N \geq 1$,
\begin{equation*}
        \E \left[ \| \mathcal{G}_{\setX}(\vx^R, \nabla F(\vx^R), \beta) \|^2\right] \leq \frac{\mathcal{C}_3}{\sqrt{N}}.
\end{equation*}
\end{itemize}
All expectations are taken with respect to all random sequences generated by the respective algorithms and an independent random integer number $R$ uniformly distributed over $\{1,\dots,N\}$. In both cases, the number of calls to SFO and LMO to get an $\epsilon$-stationary point is upper bounded by $\mathcal{O}(\epsilon^{-2}), \mathcal{O}(\epsilon^{-3})$ respectively.
\end{theorem}

\begin{remark}\label{rmk:parameterfree}
While we can obtain the same complexities without using the linear approximation of the inner function for $T=2$, it seems necessary to have a parameter-free algorithm as the choice of $\beta$ in \eqref{thm:parameter-value-2-level} depends on the knowledge of the problem parameters. Indeed, the linearization term in \eqref{eq:update-u-T-level} helps use to better exploit the Lipschitz smoothness of the gradients get an error bound in the order of $\tau_k^2 \|d^k\|^2$ for estimating the inner function values. Without this term, we are only able to use the Lipschitz continuity of the inner functions and so the error estimate will increase to the order of $\tau_k \|d^k\|$. Hence, we need to choose a larger $beta$ (as in \eqref{thm:parameter-value-2-level}) to reduce $\|d^k\|$ and handle the error term without compromising the complexities. However, this is not the case for $T=1$ as it can be seen as a two-level problem whose inner function is exactly known (the identity map). In this case, the choice of $\beta$ is independent of the problem parameters with or without the linearization term.

%(see~\eqref{thm:parameter-value-2-level}) thus highlighting the benefit of linearization. The situation is different for $T=1$, where even without linearization, the choice of $\beta$ is independent of the problem parameters. \todo{more discussion??}
\end{remark}

\subsection{High-Probability Convergence for $T=1$}
In this subsection, we establish an oracle complexity result with high-probability for the case of $T=1$. We first provide a notion of $(\epsilon,\delta)$-stationary point and a related tail assumption on the stochastic first-order oracle below.
\begin{definition}\label{def:gradmapping-highprob}
A point $\bar{\vx}\in\setX$ generated by an algorithm for solving~\eqref{eq:opt-multi} is called an $(\epsilon,\delta)$-stationary point, if we have $\| \mathcal{G}_{\setX}(\bar{\vx}, \nabla F(\bar{\vx}), \beta)\|^2\leq \epsilon$  with probability $1-\delta$.
\end{definition}

\begin{assumption}\label{aspt:oracle-highprob-1-level}
Let $\Delta^{k+1} = \nabla F(x^k) - J_1^{k+1}$ for $k\geq 0$. For each $k$, given $\setF_k$ we have $\E[\Delta^{k+1}|\setF_k]=0$ and $\|\Delta^{k+1}\|\big|\setF_k$ is $K$-sub-Gaussian.
\end{assumption}

The above assumption is commonly used in the literature; see \cite{hazan2014beyond, harvey2019tight, li2020high, zhou2018convergence}. We also refer to \cite{vershynin2018high} and Appendix \ref{sec:proof-high-prob} for additional details. The high-probability bound for solving non-convex constrained problems by Algorithm \ref{alg:CG-ASA} is given below.

\begin{theorem}\label{thm:one-level-highprob}
Let Assumptions \ref{aspt:constraint}, \ref{aspt:lipschitz-T-level}, \ref{aspt:oracle-highprob-1-level} be satisfied by the optimization problem~\eqref{eq:opt-multi} with $T=1$. Let $\tau_0 = 1, t_0 = 0$, $\tau_k = \frac{1}{\sqrt{N}}, t_k = \lceil\sqrt{k}\rceil, \forall k\geq1$, where $N$ is the total number of iterations. Let $T=1$ and let $\{\vx^k, \vz^k\}_{k\geq 0}$ be the sequence generated by Algorithm \ref{alg:CG-ASA} with $\beta_k \equiv\beta>0$. Then, we have $\forall N \geq 1, \delta>0$, with probability at least $1-\delta$,  
\begin{equation*}
    \min_{k=1,\dots, N} \left\| \mathcal{G}_{\setX}(\vx^k, \nabla F(\vx^k), \beta) \right\|^2 \leq \mathcal{O}\left(\frac{ K^2 \log (1/\delta)}{\sqrt{N}}\right)
\end{equation*}
Therefore, the number of calls to SFO and LMO to get an $(\epsilon,\delta)$-stationary point is upper bounded by $\mathcal{O}(\epsilon^{-2}\log^2 (1/\delta)), \mathcal{O}(\epsilon^{-3}\log^3 (1/\delta))$ respectively.
\end{theorem}
\begin{remark}
To the best of our knowledge, the above result is (\textrm{i}) the first high-probability bound for one-sample stochastic conditional gradient-type algorithm for the case of $T=1$, and (\textrm{ii}) the first high-probability bound for constrained stochastic optimization algorithms in the non-convex setting; see Appendix J of \cite{madden2021high}.
\end{remark}

\section{Proof Sketch of Main Results}\label{sec:sketch}

In this section, we only present the proof sketch. The complete proofs are provided in the appendix. For convenience, let $\vu_{T+1} = \vx$, and we denote $H_k$ as the function value of the subproblem at step $k$, $\vy^k$ as the optimal solution of the subproblem i.e.,
\begin{equation}\label{eq:definition-etak-yk}
    H_k(\vy) := H(\vy; \vx^k, \vz^k, \beta_k), \quad\vy^{k} = \underset{\vy\in\setX}{\argmin}~H_k(\vy).
\end{equation}
Then, the proof of Theorem \ref{thm:main-T-level} proceeds via the following steps:
\begin{enumerate}[leftmargin=0.2in]
    \item We first leverage the merit function $W_k:= W_{\alpha, \gamma}(\vx^k,\vz^k,\vu^k)$ defined in \eqref{eq:definition-merit-fucntion-T-level} with appropriate choices of $\alpha, \gamma$ for any $\beta>0$ to obtain
\begin{equation*}
\begin{split}
  W_{k+1} - W_{k} \leq& -\frac{\tau_k}{2} \left( \beta \left[\| \vd^k \|^2  + \sum_{i=1}^{T} \| \vf_{i}(\vu_{i+1}^k) - \vu_{i}^k \|^2  \right]  +\frac{\beta}{20L_{\nabla F}^2}\| \nabla F(\vx^k) -\vz^k\|^2 \right) \\
  &+ \mathbf{R}_k + \tau_k\left(\frac{12}{5} + \frac{20L_{\nabla F}^2}{\beta^2}\right)\left( H_k(\tilde{\vy}^k) - H_k(\vy^k)\right), \quad \forall k\geq 0
\end{split}
\end{equation*}
where $\mathbf{R}_k$ is the residual term (see \eqref{eq:lemma-definition-Rk-T-level}) and $\E[\mathbf{R}_k | \setF_k] \leq \hat{\sigma}^2 \tau_k^2$, as shown in Proposition \ref{prop:Rk}.
\item Telescoping the above inequality, in Lemma \ref{lem:merit-function-T-level} we obtain the following:
\begin{equation*}
\begin{split}
   &\sum_{k=1}^{N}\tau_k \left[ \beta \left(\| \vd^k \|^2  + \sum_{i=1}^{T} \| \vf_{i}(\vu_{i+1}^k) - \vu_{i}^k \|^2  \right)  +\frac{\beta}{20 L_{\nabla F}^2}\| \nabla F(\vx^k) -\vz^k\|^2 \right]\\
   &\leq 2W_0 + 2\sum_{k=0}^{N} \mathbf{R}_k + \left(\frac{24}{5} + \frac{40L_{\nabla F}^2}{\beta^2}\right)\sum_{k=0}^{N}\tau_k\left( H_k(\tilde{\vy}^k) - H_k(\vy^k)\right), \quad \forall N\geq 1.
\end{split}
\end{equation*}
\item To further control the error term $H_k(\tilde{\vy}^k) - H_k(\vy^k)$ introduced by the \texttt{ICG} method, we set $t_k$, the number of iterations in \texttt{ICG} method at step $k$, to $\lceil \sqrt{k}\rceil$. By Lemma \ref{lem:fw}, we therefore have
\begin{equation*}
    H_k(\tilde{\vy}^k) - H_k(\vy^k)\leq \frac{2\beta D_{\setX}^2 (1+\delta)}{t_k+2} \leq \frac{2\beta D_{\setX}^2 (1+\delta)}{\sqrt{k}}, \quad \forall k\geq 1.
\end{equation*}
Also, with the choice of $\tau_k=\frac{1}{\sqrt{N}}$ and $\vz^0=0$, we can conclude that
\begin{equation*}
    \sum_{k=0}^{N}\tau_k\left( H_k(\tilde{\vy}^k) - H_k(\vy^k)\right) \leq  \frac{2\beta D_{\setX}^2 (1+\delta)}{\sqrt{N}}\sum_{k=1}^{N} \frac{1}{\sqrt{k}} \leq 4\beta D_{\setX}^2 (1+\delta).
\end{equation*}
\item Then, taking expectation of both sides and by the definition of random integer $R$, we have 
\begin{equation*}
\begin{split}
    \E\left[ \beta \left(\| \vd^R \|^2  + \sum_{i=1}^{T} \| \vf_{i}(\vu_{i+1}^R) - \vu_{i}^R \|^2  \right)  + \frac{\beta}{20 L_{\nabla F}^2}\| \nabla F(\vx^R) -\vz^R\|^2 \right]   \leq 2W_0 + \mathcal{B},
\end{split}
\end{equation*}
$\forall N \geq 1$, where $\mathcal{B}$ is a constant depending on the problem parameters $(\beta, \sigma^2, L, D_{\setX}, T, \delta)$.
\item As a result, we can obtain \eqref{eq:thm-gradient-mapping} and \eqref{eq:thm-function-value-error} by noting that $\forall k\geq 1$
\begin{equation*}
\begin{split}
  \| \mathcal{G}(\vx^k, \nabla F(\vx^k), \beta) \|^2 &\leq  2\beta^2 \| \vd^k \|^2 +2\beta^2\left\| \proj_{\setX}\left(\vx^k- \frac{1}{\beta}\nabla F(\vx^k)\right)-\proj_{\setX}\left(\vx^k- \frac{1}{\beta}\vz^k\right)\right\|^2\\
  &\leq 2\beta^2 \| \vd^k \|^2 + 2\|\nabla F(\vx^k)-\vz^k\|^2.
\end{split}
\end{equation*}
where the second inequality follows the non-expansiveness of the projection operator.
\end{enumerate}
The proofs of Theorems \ref{thm:two-level} and \ref{thm:one-level-highprob} follow the same argument with appropriate modifications. The high-probability convergence proof of Theorem \ref{thm:one-level-highprob} mainly consists of controlling the tail probability of the residual term $\mathbf{R}_k$ being large.

\section{Discussion}\label{sec:disc}
In this work, we propose and analyze projection-free conditional gradient-type algorithms for constrained stochastic multi-level composition optimization of the form in~\eqref{eq:opt-multi}. We show that the oracle complexity of the proposed algorithms is level-independent in terms of the target accuracy. Furthermore, our algorithm does not require any increasing order of mini-batches under standard unbiasedness and bounded second-moment assumptions on the stochastic first-order oracle, and is parameter-free. Some open questions for  future research: (\texttt{i}) Considering the one-sample setting, either improving the LMO complexity from $\mathcal{O}(\epsilon^{-3})$ to $\mathcal{O}(\epsilon^{-2})$ for general closed convex constraint sets or establishing lower bounds showing that $\mathcal{O}(\epsilon^{-3})$ is necessary while keeping the SFO in the order of $\mathcal{O}(\epsilon^{-2})$, is extremely interesting; and (\texttt{ii}) Providing high-probability bounds for stochastic multi-level composition problems ($T>1$) and under sub-Gaussian or heavy-tail assumptions (as in \cite{madden2021high, lou2022beyond}) is interesting to explore.

\subsection*{Acknowledgment}
TX was partially supported by a seed grant from the Center for Data Science and Artificial Intelligence Research,  UC Davis and National Science Foundation (NSF) grant CCF-1934568. KB was partially supported by a seed grant from the Center for Data Science and Artificial Intelligence Research, UC Davis and NSF grant DMS-2053918. SG was partially supported by an NSERC Discovery Grant. 

%    \item Extending our methodology to the case of multi-stage or min-max optimization problems is feasible and interesting.  
%    \item Establishing oracle complexities in the non-smooth setting by relaxing the smoothness assumptions made in this work would be also interesting.

\bibliographystyle{alpha}
\bibliography{citation}

\clearpage
\appendix

\begin{center}
\hrule height 4pt
\vskip 0.25in
    {\Large\bf Supplementary Materials}
\vskip 0.29in
\hrule
\end{center}

The supplementary materials are organized as follows. Appendix \ref{sec:eg} provides motivating examples for stochastic multilevel optimization. Appendix \ref{sec:tech-lemma} introduces the essential technical lemmas to complete the proof. We present the whole proofs of Theorem \ref{thm:main-T-level} and Theorem \ref{thm:two-level} in Appendix \ref{sec:proof-thm-T} and \ref{sec:proof-thm-one-two}. Finally, we present the high-probability convergence analysis particularly for the case when $T=1$ in Appendix \ref{sec:proof-high-prob}.

\section{Motivating Examples}\label{sec:eg}
Problems of the form in~\eqref{eq:opt-multi} are generalizations of the standard constrained stochastic optimization problem which is obtained when $T=1$, and arise in several machine learning applications. Some examples include sparse additive modeling in non-parametric statistics~\cite[Section 4.1]{wang2017stochastic}, Bayesian optimization~\cite{astudillo2021bayesian}, model-agnostic meta-learning~\cite{chen2021solving,fallah2021generalization}, distributionally robust optimization~\cite{qi2021online}, training graph neural networks~\cite{cong2020minimal}, reinforcement learning~\cite[Setion 1.1]{wang2016accelerating} and AUPRC maximization~\cite{qi2021stochastic, wang2022momentum,qiu2022large}. Below, we provide a concrete motivating example from the field of risk-averse stochastic optimization~\cite{ruszczynski2006optimization}. 

The mean-deviation risk-averse optimization is given by the following optimization problem
\begin{align*}%\label{eq:mdriskaverse}
\max_{x\in \setX}\left\{ \E[U(x,\xi)] -  \rho \E \left[ \{\E[U(x,\xi)]  - U(x,\xi) \}^2  \right]^{1/2} \right\}.
\end{align*}
As noted by~\cite{yang2019multi-level},~\cite{ruszczynski2021stochastic} and ~\cite{balasubramanian2020stochastic}, the above problem is a  stochastic 3-level composition optimization problem with
\begin{align*}
f_3 := \E[U(x,\xi)] \qquad f_2(z, x) := \E[\{z -U(x,\xi)\}^2] \qquad f_1((y_1,y_2)) := y_1 - \sqrt{y_2+\delta}.
\end{align*}
Here, $\delta>0$ is added to make the square root function smooth. In particular, we consider a semi-parametric data generating process given by a sparse single-index model of the form $b = g(\langle a, x^* \rangle) +\zeta$, where $g:\mathbb{R} \to \mathbb{R}$ is called the link function, $x^* \in \mathbb{R}^d$ is assumed to be a sparse vector and $\langle\cdot,\cdot\rangle$ represents the Euclidean inner-product between two vectors. Such single-index models are widely used in statistics, machine learning and economics~\cite{ruppert2003semiparametric}. A standard choices of the link function $g$ is the square function, in which case, the model is also called as the sparse phase retrieval model~\cite{wang2017solving}. Here, $a$ is the input data which is assumed to be independent of the noise $\zeta$. In this case, $\xi := (a,b)$ and the if we consider the squared-loss, then $U(x,\xi):= (b - (\langle a, x\rangle)^2)^2$ and is non-convex in $x$. The goal is to estimate the sparse index vector $x^*$ in a risk-averse manner, as they are well-known to provide stable solutions~\cite{yang2019multi-level}. To encourage sparsity, the set $\setX$ is the $\ell_1$ ball~\cite{jaggi2013revisiting}.

\section{Technical Lemmas}\label{sec:tech-lemma}

\begin{lemma}\label{lem:F-smooth}{\normalfont (Smoothness of Composite Functions \cite{balasubramanian2020stochastic})}
Assume that Assumption \ref{aspt:lipschitz-T-level} holds.
\begin{itemize}[leftmargin=2em]
    \item [a)] Define $F_i(\vx) = \vf_i \circ \vf_{i+1} \circ \cdots \circ \vf_{T} (\vx)$. Under , the gradient of $F_i$ is Lipschitz continuous with the constant
\begin{equation*}
    L_{\nabla F_i} = \sum_{j=i}^{T} \left[ L_{\nabla f_j} \prod_{l=i}^{j-1} L_{f_l} \prod_{l=j+1}^T L_{f_l}^2 \right].
\end{equation*}
    
    \item [b)] Define
\begin{equation}\label{eq:definition-Cj}
    \begin{split}
        &R_1 = L_{\nabla \vf_1} L_{\vf_2} \cdots L_{\vf_T}, \qquad R_j = L_{\vf_{1}}\cdots L_{\vf_{j-1}} L_{\nabla \vf_j} L_{\vf_{j+1}} \cdots L_{\vf_{T}} / L_{\vf_j}, \quad 2\leq j\leq T-1,\\
        &C_2 = R_1, \qquad C_j=\sum_{i=1}^{j-2} R_i \left( \prod_{l=i+1}^{j-1} L_{\vf_l} \right), \quad 3\leq j\leq T
    \end{split}
\end{equation}
and let $\vu_{T+1} = \vx$. Then, for $T\geq 2$, we have
\begin{equation}\label{lem:gradient-chain-diff-T-level}
    \left\| \nabla F(\vx) - \prod_{i=1}^{T} \nabla \vf_{T+1-i} (\vu_{T+2-i}) \right\|\leq \sum_{j=2}^{T} C_j \| \vf_j(\vu_{j+1}) - \vu_{j} \|.
\end{equation}
\end{itemize}

\end{lemma}

\begin{lemma}\label{lem:eta-smooth}{\normalfont (Smoothness of $\eta(\cdot, \cdot)$~\cite{ghadimi2020single})}
For fixed $\beta>0$ and, $\eta(\vx, \vz)$ defined in \eqref{eq:definition-eta}, 
the gradient of $\eta(\vx,\vz)$ w.r.t. $(\vx, \vz)$ is Lipschitz continuous with the constant $L_{\nabla\eta}=2 \sqrt{(1+\beta)^{2}+\left(1+\frac{1}{2 \beta}\right)^{2}}$.
\end{lemma}

\begin{lemma}\label{lem:fw} {\normalfont (Convergnece of \texttt{ICG}~\cite{jaggi2013revisiting})}
Let $\tilde{\vy}^k$ be the vector output by Algorithm \ref{alg:ICG} at step $k$, and $\vy^k$ be the optimal solution of the subproblem \ref{eq:definition-etak-yk}, then under Assumption \ref{aspt:constraint}
\begin{equation*}
        \frac{\beta}{2}\| \tilde{\vy}^k - \vy^k \|^2\leq H_k(\tilde{\vy}^k) - H_k(\vy^k)\leq \frac{2\beta D_{\setX}^2 (1+\delta)}{t_k+2}\\
\end{equation*}
where $\delta$ defined in Algorithm \ref{alg:ICG} is the quality of the linear minimization procedure.
\end{lemma}
\begin{proof}[Proof of Lemma \ref{lem:fw}]
The result is obtained by applying Theorem 1 in \cite{jaggi2013revisiting} to $H_k$ and noting that the curvature constant $C_{H_k} = \beta D_{\setX}^2, \forall k\geq0$.
\end{proof}

\section{Proof of Theorem \ref{thm:main-T-level}}\label{sec:proof-thm-T}
To establish the rate of convergence for Algorithm \ref{alg:CG-NASA-T-level} in Theorem \ref{thm:main-T-level}, we first present Lemma \ref{lem:basic-ineq-T-level} and Lemma \ref{lem:mse-uk-diff-order} regarding the basic recursion on the errors in estimating the inner function values and the order of $\E[\| \vu_{i}^{k+1} - \vu_{i}^{k}\|^2|\setF_k]$. The proofs follow \cite{balasubramanian2020stochastic} with minor modifications. We present the complete proofs below for the reader's convenience.  

\begin{lemma}\label{lem:basic-ineq-T-level}
Let $\{\vx^k\}_{k\geq 0}$ and $\{\vu_{i}^{k}\}_{k\geq 0}$ be generated by Algorithm \ref{alg:CG-NASA-T-level} and $\vu_{T+1} = \vx$. Define, $1\leq i\leq T$,
\begin{equation}\label{eq:definition-T-level-Delta-e}
\begin{split}
    \Delta_{\vG_i}^{k+1} &:= \vf_{i}(\vu_{i+1}^k) - \vG_{i}^{k+1},  \quad \Delta_{\vJ_i}^{k+1} := \nabla \vf_i(\vu_{i+1}^k) - \vJ_{i}^{k+1},\\
    \ve_{i}^{k} &:= \vf_{i}(\vu_{i+1}^{k+1}) - \vf_{i}(\vu_{i+1}^{k}) - \langle \nabla \vf_{i}(\vu_{i+1}^k), \vu_{i+1}^{k+1} - \vu_{i+1}^{k}\rangle.
\end{split}
\end{equation}
Under Assumption \ref{aspt:lipschitz-T-level}, we have, for $1\leq i\leq T$,
\begin{equation}\label{eq:lemma-T-level-recursion-function-values-error}
\begin{split}
    &\| \vf_{i} (\vu_{i+1}^{k+1}) - \vu_{i}^{k+1} \|^2 \leq (1-\tau_k) \| \vf_i(\vu_{i+1}^k) - \vu_{i}^{k} \|^2 + \tau_k^2 \| \Delta_{\vG_i}^{k+1} \|^2  + \dot{r}_{i}^{k+1}\\
    &\quad + \left[ 4 L_{\vf_i}^2  +  L_{\nabla \vf_i} \| \vf_i(\vu_{i+1}^k) - \vu_{i}^{k} \| + \| \Delta_{\vJ_i}^{k+1} \|^2\right]\| \vu_{i+1}^{k+1} - \vu_{i+1}^{k} \|^2,
\end{split}
\end{equation}
and
\begin{equation}\label{eq:lemma-T-level-function-values-diff}
\begin{split}
    \| \vu_{i}^{k+1} - \vu_{i}^{k} \|^2 \leq \tau_k^2 \left[2\| \vf_{i}(\vu_{i+1}^{k}) - \vu_{i}^{k}\|^2  + \| \Delta_{\vG_i}^{k+1}\|^2\right] + 2\|\vJ_{i}^{k+1}\|^2\| \vu_{i+1}^{k+1} - \vu_{i+1}^{k} \|^2 + \ddot{r}_{i}^{k+1}
\end{split}
\end{equation}
where
\begin{equation}\label{eq:defnition-T-level-dot-r}
    \begin{split}
        &\dot{r}_{i}^{k+1} := 2\tau_k \langle \Delta_{\vG_i}^{k+1}, \ve_{i}^{k} + (1-\tau_k) (\vf_{i}(\vu_{i+1}^{k}) - \vu_{i}^{k}) + {\Delta_{\vJ_i}^{k+1}}^{\top} (\vu_{i+1}^{k+1} - \vu_{i+1}^{k}) \rangle\\
        &\qquad\qquad + 2 \langle {\Delta_{\vJ_i}^{k+1}}^{\top}(\vu_{i+1}^{k+1} - \vu_{i+1}^{k}), \ve_{i}^{k} + (1-\tau_k) (\vf_i(\vu_{i+1}^{k}) - \vu_{i}^{k}) \rangle,\\
        &\ddot{r}_{i}^{k+1} := \tau_{k} \langle -\Delta_{\vG_i}^{k+1}, \tau_{k}(\vf_i(\vu_{i+1}^k) - \vu_{i}^{k}) + {\vJ_{i}^{k+1}}^{\top}(\vu_{i+1}^{k+1} - \vu_{i+1}^k) \rangle.
    \end{split}
\end{equation}
\end{lemma}

\begin{proof}
We first prove part \eqref{eq:lemma-T-level-recursion-function-values-error}. By the definitions in \eqref{eq:definition-T-level-Delta-e}, \eqref{eq:defnition-T-level-dot-r}, for any $1\leq i\leq T$, we have
\begin{equation*}
    \begin{split}
        &\| \vf_i(\vu_{i+1}^{k+1}) - \vu_{i}^{k+1} \|^2\\
        =& \|\ve_{i}^{k} + \vf_{i}(\vu_{i+1}^{k}) + \nabla \vf_{i}(\vu_{i+1}^{k})^\top(\vu_{i+1}^{k+1} - \vu_{i+1}^{k}) - (1-\tau_k) \vu_{i}^{k} - \tau_k \vG_{i}^{k+1} - {\vJ_{i}^{k+1}}^\top(\vu_{i+1}^{k+1} - \vu_{i+1}^{k})\|^2\\
        =& \|\ve_{i}^{k} + {\Delta_{\vJ_i}^{k+1}}^\top (\vu_{i+1}^{k+1} - \vu_{i+1}^{k}) + (1-\tau_k) (\vf_{i}(\vu_{i+1}^{k}) - \vu_{i}^{k}) + \tau_k \Delta_{\vG_i}^{k+1} \|^2\\
        =& \|{\Delta_{\vJ_i}^{k+1}}^\top (\vu_{i+1}^{k+1} - \vu_{i+1}^{k}) \|^2 + \| \ve_{i}^{k}+(1-\tau_k) (\vf_{i}(\vu_{i+1}^{k}) - \vu_{i}^{k}) \|^2 + \tau_k^2 \| \Delta_{\vG_i}^{k+1} \|^2 + \dot{r}_{i}^{k+1}\\
        \leq&  \|  \ve_{i}^{k}+(1-\tau_k) (\vf_{i}(\vu_{i+1}^{k}) - \vu_{i}^{k}) \|^2 + \tau_k^2 \| \Delta_{\vG_i}^{k+1}\|^2 + \|\Delta_{\vJ_i}^{k+1}\|^2\|\vu_{i+1}^{k+1} - \vu_{i+1}^{k} \|^2 + \dot{r}_{i}^{k+1}\\
        \leq& (1-\tau_k) \| \vf_{i}(\vu_{i+1}^{k}) - \vu_{i}^k \|^2 + \| \ve_{i}^{k} \|^2 + 2(1-\tau_k) \| \ve_{i}^{k} \| \|\vf_{i}(\vu_{i+1}^k) - \vu_{i}^{k}\| + \tau_k^2 \| \Delta_{\vG_i}^{k+1}\|^2 \\
        &\qquad + \|\Delta_{\vJ_i}^{k+1}\|^2\|\vu_{i+1}^{k+1} - \vu_{i+1}^{k} \|^2 + \dot{r}_{i}^{k+1}.
    \end{split}
\end{equation*}
Furthermore, with Assumption \ref{aspt:lipschitz-T-level}, we have
\begin{equation}\label{eq:lemma-T-level-e-bound}
    \| \ve_{i}^k \| \leq \frac{L_{\nabla \vf_i}}{2}\| \vu_{i+1}^{k+1} - \vu_{i+1}^{k} \|^2 , \qquad \| \ve_{i}^{k}\|^2 \leq 4 L_{\vf_i}^2 \| \vu_{i+1}^{k+1} - \vu_{i+1}^{k} \|^2,
\end{equation}
which leads to \eqref{eq:lemma-T-level-recursion-function-values-error}. To show \eqref{eq:lemma-T-level-function-values-diff}, with the update rule given by \eqref{eq:update-u-T-level} and the definitions in \eqref{eq:definition-T-level-Delta-e}, we have, for $1\leq i\leq T$,
\begin{equation*}
\begin{split}
    &\| \vu_{i}^{k+1} - \vu_{i}^{k} \|^2\\
    =& \| \tau_k (\vG_{i}^{k+1} - \vu_{i}^{k}) + \langle \vJ_{i}^{k+1}, \vu_{i+1}^{k+1} - \vu_{i+1}^{k}\rangle \|^2\\
    =& \tau_k^2 \| \vG_{i}^{k+1} - \vu_i^{k} \|^2 + \| {\vJ_{i}^{k+1}}^\top (\vu_{i+1}^{k+1} - \vu_{i+1}^{k}) \|^2  + 2\tau_k \langle \vG_{i}^{k+1} - \vu_i^{k}, {\vJ_{i}^{k+1}}^\top(\vu_{i+1}^{k+1} - \vu_{i+1}^{k}) \rangle\\
    =& \tau_k^2 \| \vG_{i}^{k+1} - \vu_i^{k} \|^2 + \| {\vJ_{i}^{k+1}}^\top(\vu_{i+1}^{k+1} - \vu_{i+1}^{k}) \|^2 + 2\tau_k\langle \vf_i(\vu_{i+1}^{k}) - \vu_{i}^{k}, {\vJ_{i}^{k+1}}^\top(\vu_{i+1}^{k+1} - \vu_{i+1}^{k}) \rangle \\
    &\qquad + 2\tau_k\langle -\Delta_{\vG_i}^{k+1}, {\vJ_{i}^{k+1}}^\top(\vu_{i+1}^{k+1} - \vu_{i+1}^{k}) \rangle\\
    \leq&\tau_k^2 \| \vG_{i}^{k+1} - \vu_i^{k} \|^2 + 2\| \vJ_{i}^{k+1}\|^2\|\vu_{i+1}^{k+1} - \vu_{i+1}^{k} \|^2 + \tau_k^2 \| \vf_i(\vu_{i+1}^{k}) - \vu_{i}^{k}\|^2\\
     &\qquad + 2\tau_k\langle -\Delta_{\vG_i}^{k+1}, {\vJ_{i}^{k+1}}^\top(\vu_{i+1}^{k+1} - \vu_{i+1}^{k}) \rangle\\
    =& 2\tau_k^2 \| \vf_i(\vu_{i+1}^k)  - \vu_{i}^{k}\|^2 + \tau_{k}^2 \| \Delta_{\vG_i}^{k+1}\|^2 + 2 \|\vJ_{i}^{k+1}\|^2 \| \vu_{i+1}^{k+1} - \vu_{i+1}^{k} \|^2\\
    &\qquad + 2\tau_k \langle -\Delta_{\vG_i}^{k+1}, \tau_k(\vf_i(\vu_{i+1}^k) - \vu_{i}^{k}) + {\vJ_{i}^{k+1}}^\top(\vu_{i+1}^{k+1} - \vu_{i+1}^k)\rangle.
\end{split}
\end{equation*}
where the inequality comes from the fact that $\| {\vJ_{i}^{k+1}}^\top(\vu_{i+1}^{k+1} - \vu_{i+1}^{k}) \|^2\leq \| \vJ_{i}^{k+1}\|^2\|\vu_{i+1}^{k+1} - \vu_{i+1}^{k}\|^2$ and $2\tau_k\langle \vf_i(\vu_{i+1}^{k}) - \vu_{i}^{k}, {\vJ_{i}^{k+1}}^\top(\vu_{i+1}^{k+1} - \vu_{i+1}^{k}) \rangle\leq \| {\vJ_{i}^{k+1}}^\top(\vu_{i+1}^{k+1} - \vu_{i+1}^{k})\|^2 + \tau_k^2 \| \vf_i(\vu_{i+1}^{k}) - \vu_{i}^{k}\|^2$.
\end{proof}

\begin{lemma}\label{lem:mse-uk-diff-order}
Let $\vu_{T+1} = \vx$. Under Assumption \ref{aspt:lipschitz-T-level}, \ref{aspt:oracle-T-level}, and with the choice of $\tau_0=1$, we have, for $1\leq i\leq T$ and $k\geq 0$,
\begin{align}
    &\E[\| \vf_i(\vu_{i+1}^{k+1}) - \vu_{i}^{k+1} \|^2|\setF_k] \leq \sigma_{\vG_i}^2 + (4L_{\vf_i}^2 + \sigma_{\vJ_i}^2) c_{i+1},\label{eq:lemma-function-value-error}\\
    &\E[\| \vu_{i}^{k+1} - \vu_{i}^{k}\|^2|\setF_k]\leq c_i \tau_k^2,\label{eq:lemma-function-value-diff}
\end{align}
where
\begin{equation}
    c_i := 3\sigma_{\vG_i}^2 + 2(4L_{\vf_i}^2 + \sigma_{\vJ_i}^2 + \hat{\sigma}_{\vJ_i}^2) c_{i+1}, \qquad c_{T+1} = D_{\setX}^2.
\end{equation}

\end{lemma}
\begin{proof}
By the update rule given in \eqref{eq:update-u-T-level} and the definitions in \eqref{eq:definition-T-level-Delta-e}, for $1\leq i \leq T$ and $k\geq 0$, we have
\begin{equation*}
    \vf_i(\vu_{i+1}^{k+1}) - \vu_{i}^{k+1} = (1-\tau_k) (\vf_{i}(\vu_{i+1}^k) - \vu_{i}^k) + \mathbf{D}_{k,i},
\end{equation*}
where $\mathbf{D}_{k, i} := \ve_{i}^{k} + \tau_k \Delta_{\vG_i}^{k+1} + {\Delta_{\vJ_i}^{k+1}}^\top (\vu_{i+1}^{k+1} - \vu_{i+1}^{k})$. With the convexity of $\| \cdot \|^2$, we can further obtain
\begin{equation}\label{eq:lemma-function-value-error-bound}
    \|\vf_i(\vu_{i+1}^{k+1}) - \vu_{i}^{k+1}\|^2 \leq (1-\tau_k) \|  \vf_{i}(\vu_{i+1}^k) - \vu_{i}^k\|^2 +\frac{1}{\tau_k} \| \mathbf{D}_{k,i} \|^2,\quad \forall k\geq 0.
\end{equation}
Moreover, under Assumption \ref{aspt:oracle-T-level}, we have, for $1\leq i \leq T$ and $k\geq 0$,
\begin{equation}\label{eq:lemma-Dki-bound}
\begin{split}
    \E[\| \mathbf{D}_{k,i} \|^2 | \setF_k] &= \E[\|\ve_{i}^{k}\|^2 | \setF_k] + \tau_k^2 \E [\|\Delta_{\vG_i}^{k+1}\|^2|\setF_k] + \E[\|{\Delta_{\vJ_i}^{k+1}}^\top(\vu_{i+1}^{k+1}-\vu_{i+1}^k)\|^2|\setF_k]\\
    &\leq \tau_k^2  \E [\|\Delta_{\vG_i}^{k+1}\|^2|\setF_k] + \left(4 L_{\vf_i}^2 + \E[\| \Delta_{\vJ_i}^{k+1} \|^2|\setF_k]\right) \E[\|\vu_{i+1}^{k+1} - \vu_{i+1}^{k} \|^2|\setF_k]\\
    &\leq \tau_k^2  \sigma_{\vG_i}^2 + \left(4 L_{\vf_i}^2 + \sigma_{\vJ_i}^2\right) \E[\|\vu_{i+1}^{k+1} - \vu_{i+1}^{k} \|^2|\setF_k].
\end{split}
\end{equation}
where the second inequality follows from \eqref{eq:lemma-T-level-e-bound}. Setting $i=T$ in the inequality above and noting that $\vu_{T+1}^k = \vx^k$, we have
\begin{equation*}
    \E[\| \mathbf{D}_{k,T}\|^2|\setF_k] \leq \tau_k^2\left[\sigma_{\vG_T}^2 + (4L_{\vf_T}^2 +\sigma_{\vJ_T}^2) D_{\setX}^2 \right], \quad \forall k\geq 0.
\end{equation*}
Thus, with the choice of $\tau_0=1$, we obtain
\begin{equation*}
    \E[\|\vf_{T}(\vx^{k}) - \vu_{T}^{k}\|^2|\setF_k] \leq  \sigma_{\vG_T}^2 + (4L_{\vf_T}^2 +\sigma_{\vJ_T}^2) D_\setX^2,\quad \forall k \geq 1.
\end{equation*}
Taking expectation of both sides of \eqref{eq:lemma-T-level-function-values-diff} conditioning on $\setF_k$, and under Assumption \ref{aspt:oracle-T-level}, we obtain
\begin{equation}\label{eq:lemma-function-value-diff-bound}
\begin{split}
     \E[\| \vu_{i}^{k+1} - \vu_{i}^{k} \|^2|\setF_k] \leq \tau_k^2 \E\left[2\| \vf_{i}(\vx^{k}) - \vu_{i}^{k}\|^2  + \| \Delta_{\vG_i}^{k+1}\|^2+\frac{2}{\tau_k^2}\|\vJ_{i}^{k+1}\|^2\| \vu_{i+1}^{k+1} - \vu_{i+1}^{k} \|^2\middle|\setF_k\right].
\end{split}
\end{equation}
Setting $i = T$ in the inequality above, we have
\begin{equation*}
\begin{split}
     \E[\| \vu_{T}^{k+1} - \vu_{T}^{k} \|^2|\setF_k] \leq \tau_k^2 \left[3\sigma_{\vG_T}^2 + 2(4L_{\vf_T}^2 +\sigma_{\vJ_T}^2 + \hat{\sigma}_{\vJ_T}^2) D_\setX^2 \right], \quad \forall k\geq 1.
\end{split}
\end{equation*}
This completes the proof of \eqref{eq:lemma-function-value-error} and \eqref{eq:lemma-function-value-diff} when $i = T$. We now use backward induction to complete
the proof. By the above result, the base case of $i = T$ holds. Assume that \eqref{eq:lemma-function-value-diff} hold when $i = j$ for some $1< j\leq T$, i.e., $\E[\| \vu_{j}^{k+1} - \vu_{j}^{k}\|^2|\setF_k]\leq c_j \tau_k^2, \forall k \geq 0$. Then, setting $i=j-1$ in \eqref{eq:lemma-Dki-bound}, we obtain
\begin{equation*}
    \E[\|\mathbf{D}_{k,j-1} \|^2|\setF_k] \leq \tau_k^2\left[\sigma_{\vG_{j-1}}^2 + (4 L_{\vf_{j-1}}^2 + \sigma_{\vJ_{j-1}}^2)c_{j}\right], \quad \forall k\geq 0.
\end{equation*}
Furthermore, with \eqref{eq:lemma-function-value-error-bound} and the choice of $\tau_0=1$, we have
\begin{equation*}
    \E[\|\vf_{j-1}(\vu_{j}^{k+1}) - \vu_{j-1}^{k+1}\|^2 |\setF_k]\leq\sigma_{\vG_{j-1}}^2 + (4 L_{\vf_{j-1}}^2 + \sigma_{\vJ_{j-1}}^2)c_{j}, \quad \forall k\geq 0.
\end{equation*}
which together with \eqref{eq:lemma-function-value-diff-bound}, imply that
\begin{equation*}
    \E[\| \vu_{j-1}^{k+1} - \vu_{j-1}^{k}\|^2|\setF_k]\leq c_{j-1} \tau_k^2, \quad \forall k\geq 0.
\end{equation*}
\end{proof}

We now leverage the merit function defined in \eqref{eq:definition-merit-fucntion-T-level} and provide a basic inequality for establishing convergence analysis of Algorithm \ref{alg:CG-NASA-T-level} in Lemma~\ref{lem:merit-function-T-level}. In Proposition~\ref{prop:Rk}, we show the boundedness of the term $\mathbf{R}_k$ appearing on the right hand side of \eqref{eq:lemma-merit-function-telescoping-T-level} in expectation. These two results form the crucial steps in establishing the convergence analysis of Algorithm~\ref{alg:CG-NASA-T-level}.

\begin{lemma}
\label{lem:merit-function-T-level}
Let $\{\vx^k, \vz^k, \vu^k\}_{k\geq 0}$ be the sequence generated by Algorithm \ref{alg:CG-NASA-T-level}, the merit function $W_{\alpha,\gamma}(\cdot, \cdot, \cdot)$ be defined in \eqref{eq:definition-merit-fucntion-T-level} with positive constants $\{\alpha, \{\gamma_i\}_{1\leq i\leq T}\}$, and $\vu_{T+1} = \vx$. Under Assumption \ref{aspt:lipschitz-T-level}, for any $\beta > 0$, let
\begin{equation*}
    \beta_k \equiv \beta, \quad\alpha = \frac{\beta}{20 L_{\nabla F}^2},\quad \gamma_1 = \frac{\beta}{2} ,\quad \gamma_j =  \left(2\alpha + \frac{1}{4\alpha L_{\nabla F}^2}\right) (T-1)C_j^2 + \frac{\beta}{2},  \quad 2\leq j\leq T,
\end{equation*}
where $C_j$'s are defined in \eqref{eq:definition-Cj}. Then, $\forall N \geq 0$
\begin{equation}\label{eq:lemma-merit-function-telescoping-T-level}
\begin{split}
   \sum_{k=0}^{N}\tau_k \left( \beta \left[\| \vd^k \|^2  + \sum_{i=1}^{T} \| \vf_{i}(\vu_{i+1}^k) - \vu_{i}^k \|^2  \right]  +\frac{\beta}{20 L_{\nabla F}^2}\| \nabla F(\vx^k) -\vz^k\|^2 \right)\\
   \leq 2W_0 + 2\sum_{k=0}^{N} \mathbf{R}_k + \left(\frac{24}{5} + \frac{40L_{\nabla F}^2}{\beta^2}\right)\sum_{k=0}^{N}\tau_k\left( H_k(\tilde{\vy}^k) - H_k(\vy^k)\right),
\end{split}
\end{equation}
where $\vd^k :=  \vy^{k} - \vx^k$, $H_k(\cdot), \vy^k$ are defined in \eqref{eq:definition-etak-yk}, and
\begin{equation}\label{eq:lemma-definition-Rk-T-level}
\begin{split}
    \mathbf{R}_k :=& \sum_{i=1}^{T} \gamma_i  \left[ 4 L_{\vf_i}^2  +  L_{\nabla \vf_i} \| \vf_i(\vu_{i+1}^k) - \vu_{i}^{k} \| + \| \Delta_{\vJ_i}^{k+1} \|^2\right]\| \vu_{i+1}^{k+1} - \vu_{i+1}^{k} \|^2\\
   &+\tau_k^2\left[\frac{L_{\nabla F}+ L_{\nabla\eta}}{2} D_{\setX}^2  + \sum_{i=1}^{T}\gamma_i \| \Delta_{\vG_{i}}^{k+1} \|^2 + \alpha \| \Delta^{k+1} \|^2\right] \\
    &+ \tau_k \left[ \langle \vd^k , \Delta^{k+1}\rangle + \sum_{i=1}^{T}\gamma_i \dot{r}_{i}^{k+1} + 2\alpha \dddot{r}^{k+1}\right]+ \frac{L_{\nabla \eta}}{2} \| \vz^{k+1} - \vz^k \|^2,\\
    \Delta^{k+1}:= & \prod_{i=1}^{T} \nabla \vf_{T+1-i}(\vu_{T+2-i}^k) - \prod_{i=1}^{T}\vJ^{k+1}_{T-i+1},\\
    \dddot{r}^{k+1}:=&\langle \Delta^{k+1}, (1-\tau_{k}) [ \nabla F(\vx^{k}) - \vz^{k}] + \tau_{k}[\nabla F(\vx^k) -  \prod_{i=1}^{T} \nabla \vf_{T+1-i}(\vu_{T+2-i}^k)] \\
    &+  \nabla F(\vx^{k+1}) - \nabla F(\vx^k)\rangle,
\end{split}
\end{equation}
$\Delta_{\vG_i}^{k+1}$, $\Delta_{\vJ_i}^{k+1}$ are defined in \eqref{eq:definition-T-level-Delta-e}, and $\dot{r}_i^{k+1}$ is defined in \eqref{eq:defnition-T-level-dot-r}.
\end{lemma}
\begin{proof} 
We first bound $F(\vx^{k+1}) - F(\vx^k)$. By the Lipschitzness of $\nabla F$ (Lemma \ref{lem:F-smooth}), we have
\begin{equation}\label{eq:lemma-proof-smoothness-F-T-level}
    \begin{split}
        &F (\vx^{k+1}) - F(\vx^{k}) \leq \langle \nabla F(\vx^{k}), \vx^{k+1} - \vx^{k} \rangle + \frac{L_{\nabla F}\tau_k^2}{2} \| \tilde{\vd}^{k} \|^2\\
        &=\tau_k\langle \nabla F(\vx^{k}), \vd^k \rangle + \tau_k\langle\nabla F(\vx^{k}) - \vz^{k}, \tilde{\vy}^k - \vy^k \rangle- \tau_k \langle \beta \vd^k, \tilde{\vy}^k - \vy^k \rangle\\
        &\qquad+ \tau_k\langle\vz^k+\beta\vd^k, \tilde{\vy}^k - \vy^k\rangle + \frac{L_{\nabla F} \tau_k^2}{2} \| \tilde{\vd}^{k} \|^2\\
        &\leq \tau_k\langle \nabla F(\vx^{k}), \vd^k \rangle+ \tau_k\|\nabla F(\vx^{k}) - \vz^{k}\|\| \tilde{\vy}^k - \vy^k \| + \tau_k\langle\vz^k+\beta\vd^k, \tilde{\vy}^k - \vy^k\rangle \\
        &\qquad + \tau_k \beta \|\vd^k\|\| \tilde{\vy}^k - \vy^k\| +  \frac{L_{\nabla F} \tau_k^2}{2} \| \tilde{\vd}^{k} \|^2.
    \end{split}
\end{equation}

We then provide a bound for $\eta(\vx^k, \vz^k) - \eta (\vx^{k+1}, \vz^{k+1})$. By the lipschitzness of $\nabla \eta$ (Lemma \ref{lem:eta-smooth}) with the partial gradients of $\nabla \eta$ given by
\begin{equation*}
    \nabla_\vx \eta(\vx^k, \vz^k) = -\vz^k - \beta \vd^k,\quad \nabla_\vz \eta(\vx^k,\vz^k) = \vd^k,
\end{equation*}
we have
\begin{equation}\label{eq:lemma-proof-smoothness-eta}
    \begin{split}
        &\eta(\vx^k, \vz^k) - \eta (\vx^{k+1}, \vz^{k+1}) \\
        &\leq \big\langle \vz^k + \beta \vd^k, \vx^{k+1} - \vx^k\big\rangle - \big\langle \vd^k, \vz^{k+1}-\vz^k \big\rangle + \frac{L_{\nabla\eta}}{2} \bigg[\| \vx^{k+1} - \vx^k \|^2 + \| \vz^{k+1} - \vz^k \|^2\bigg]\\
        &= \tau_k \big\langle 2\vz^k + \beta \vd^k, \vd^k\big\rangle  + \tau_k \big\langle\vz^k + \beta \vd^k, \tilde{\vd}^k - \vd^k\big\rangle - \tau_k \big\langle \vd^k, \prod_{i=1}^{T} \vJ_{T-i+1}^{k+1}  \big\rangle\\
        & \qquad + \frac{L_{\nabla\eta}}{2} \bigg[\tau_k^2\| \tilde{\vd}^{k} \|^2 + \| \vz^{k+1} - \vz^k \|^2\bigg],
    \end{split}
\end{equation}
where the second equality comes from \eqref{eq:update-x-T-level} and \eqref{eq:update-z-T-level}. Due to the optimality condition of in the definition of $\vy^k$, we have $\big\langle  \vz^k +\beta \vd^k, \vx - \vy^k \big\rangle  \geq 
0$ for all $\vx \in  \setX$, which together with the choice of $\vx = \vx^k$ implies that
\begin{equation}\label{eq:lemma-proof-optimality-eta}
    \langle \vz^k, \vd^k \rangle + \beta \| \vd^k \|^2\leq 0.
\end{equation}
Thus, combining \eqref{eq:lemma-proof-smoothness-eta} with \eqref{eq:lemma-proof-optimality-eta}, we obtain
\begin{equation}\label{eq:lemma-proof-eta-diff-T-level}
    \begin{split}
        \eta (\vx^k, \vz^k) - \eta (\vx^{k+1}, \vz^{k+1})\leq  -\beta \tau_k \|\vd^k \|^2  + \tau_k \big\langle\vz^k + \beta \vd^k, \tilde{\vy}^k - \vy^k\big\rangle- \tau_k\big\langle \vd^k, \prod_{i=1}^{T} \vJ_{T-i+1}^{k+1} \big\rangle\\
        + \frac{L_{\nabla\eta}}{2} \bigg[\tau_k^2\| \tilde{\vd}^{k} \|^2 + \| \vz^{k+1} - \vz^k \|^2\bigg].
    \end{split}
\end{equation}
In addition, by Lemma \ref{lem:gradient-chain-diff-T-level}, we have
\begin{equation}\label{eq:lemma-proof-gradient-chain-diff-T-level}
    \langle \vd^k, \nabla F (\vx^k) - \prod_{i=1}^{T} \nabla \vf_{T+1-i} (\vu_{T+2-i}^k)  \rangle \leq \sum_{j=2}^{T} C_j \| \vd^k\| \| \vf_j(\vu_{j+1}^k) - \vu_{j}^k \|.
\end{equation}
Then combing \eqref{eq:lemma-proof-smoothness-F-T-level}, \eqref{eq:lemma-proof-eta-diff-T-level}, \eqref{eq:lemma-proof-gradient-chain-diff-T-level}, we have
\begin{equation}\label{eq:lemma-proof-F-eta-diff-T-level}
    \begin{split}
        &[F(\vx^{k+1}) - \eta(\vx^{k+1}, \vz^{k+1})] - [F(\vx^k) - \eta(\vx^k, \vz^k)]\\
        &\leq \tau_k \bigg\{ -\beta \| \vd^k \|^2 + \sum_{j=2}^{T} C_j \|\vd^k\|\|\vf_j(\vu_{j+1}^k) - \vu_j^k\|+ \langle \vd^k, \Delta^{k+1} \rangle +2 \langle \vz^k + \beta \vd^k, \tilde{\vy}^k - \vy^k \rangle\\
        &~+\left[\beta \|\vd^k\| + \|\nabla F(\vx^k) - \vz^k\|\right]\|\tilde{\vy}^k - \vy^k\| \bigg\} + \frac{L_{\nabla F} + L_{\nabla \eta}}{2} \tau_k^2 \| \tilde{\vd}^{k} \|^2 + \frac{L_{\nabla \eta}}{2} \| \vz^{k+1} - \vz^k \|^2.
    \end{split}
\end{equation}
Furthermore, defining
\begin{equation*}
        \varkappa^k := \nabla F(\vx^k) -  \prod_{i=1}^{T} \nabla \vf_{T+1-i}(\vu_{T+2-i}^k),\qquad \bar{ \varkappa}^k := \frac{\nabla F(\vx^{k+1}) - \nabla F(\vx^k)}{\tau_k},
\end{equation*}
and by the update rule given by \eqref{eq:update-z-T-level}, we have 
\begin{equation}\label{eq:lemma-proof-dual-recursion-T-level}
    \begin{split}
        &\| \nabla  F(\vx^{k+1}) - \vz^{k+1}  \|^2 \\
        &= \| (1-\tau_{k}) [ \nabla F(\vx^{k}) - \vz^{k}] + \tau_{k}[\varkappa^k + \bar{\varkappa}^k + \Delta^{k+1}]\|^2\\
        &= \| (1-\tau_{k}) [ \nabla F(\vx^{k}) - \vz^{k}] +  \tau_{k}[\varkappa^k + \bar{\varkappa}^k]\|^2 + \tau_{k}^2\|\Delta^{k+1}\|^2+ 2\tau_{k} \dddot{r}^{k+1}\\
        &\leq (1-\tau_{k}) \| \nabla F(\vx^{k}) -\vz^{k}\|^2 + 2\tau_{k} \left[\| \varkappa^k\|^2+ \|\bar{\varkappa}^k\|^2\right]+ \tau_{k}^2\|\Delta^{k+1}\|^2+ 2\tau_{k} \dddot{r}^{k+1}\\
        &\leq (1-\tau_{k}) \| \nabla F(\vx^{k}) -\vz^{k}\|^2 + \tau_{k}^2\|\Delta^{k+1}\|^2\\
        &+ 2\tau_{k} \left[(T-1) \sum_{j=2}^{T} C_j^2 \| \vf_{j}(\vu_{j+1}) - \vu_{j} \|^2+ 2L_{\nabla F}^2(\|\vd^k\|^2 + \|\tilde{\vy}^k- \vy^k \|^2)+\dddot{r}^{k+1}\right].\\
    \end{split}
\end{equation}
where $\dddot{r}^{k+1}:=\langle \Delta^{k+1}, (1-\tau_{k}) [ \nabla F(\vx^{k}) - \vz^{k}] + \tau_{k}[\varkappa^k + \bar{\varkappa}^k ]\rangle$ and the last inequality comes from two fact that $\|\bar{\varkappa}^k\|^2\leq 2L_{\nabla F}^2(\|\vd^k\|^2 + \|\tilde{\vy}^k- \vy^k \|^2)$ and 
\begin{equation*}
    \begin{split}
        \| \varkappa^k \|^2 = \left\| \nabla F(\vx^k) - \prod_{i=1}^{T} \nabla \vf_{T+1-i}(\vu_{T+2-i}^k) \right\|^2  \leq (T-1) \sum_{j=2}^{T} C_j^2 \| \vf_{j}(\vu_{j+1}) - \vu_{j} \|^2.
    \end{split}
\end{equation*}
The above upper bound for the term $\|\varkappa^k \|^2$ is obtained by leveraging Lemma \ref{lem:gradient-chain-diff-T-level} and the fact that $(\sum_{i=1}^{n} a_i) \leq n \sum_{i=1}^{n} a_i^2$ for non-negative sequence $(a_i)_{1\leq i \leq n}$. 

Moreover, by Lemma \ref{lem:basic-ineq-T-level}, we have, for $1\leq i \leq T$,
\begin{equation}\label{eq:lemma-proof-function-value-error-diff-T-level}
\begin{split}
    &\| \vf_{i} (\vu_{i+1}^{k+1}) - \vu_{i}^{k+1} \|^2 - \| \vf_i(\vu_{i+1}^k) - \vu_{i}^{k} \|^2 \leq -\tau_k \| \vf_i(\vu_{i+1}^k) - \vu_{i}^{k} \|^2 + \tau_k^2 \| \Delta_{\vG_i}^{k+1} \|^2  + \dot{r}_{i}^{k+1}\\
    &\qquad \qquad + \left[ 4 L_{\vf_i}^2  +  L_{\nabla \vf_i} \| \vf_i(\vu_{i+1}^k) - \vu_{i}^{k} \| + \| \Delta_{\vJ_i}^{k+1} \|^2\right]\| \vu_{i+1}^{k+1} - \vu_{i+1}^{k} \|^2,
\end{split}
\end{equation}

Finally, multiplying both sides of \eqref{eq:lemma-proof-function-value-error-diff-T-level} by $\gamma_i$ for $i=1,\dots, T$ and both sides of \eqref{eq:lemma-proof-dual-recursion-T-level} by $\alpha$, adding them to \eqref{eq:lemma-proof-F-eta-diff-T-level}, rearranging the terms, and noting that $\langle\vz^k+\beta\vd^k, \tilde{\vy}^k - \vy^k\rangle = H_k (\tilde{\vy}^k) - H_k(\vy^k) - (\beta/2)\| \tilde{\vy}^k - \vy^k \|^2$ due to the quadratic structure of $H_k$ and $\| \tilde{\vd}^{k} \|^2\leq D_{\setX}^2$, we obtain
\begin{equation}\label{eq:merit-func-diff-T-level}
    W_{k+1} -  W_k \leq \tau_k \mathbf{A}_k + \mathbf{R}_k
\end{equation}
where $\mathbf{R}_k$ is defined in \eqref{eq:lemma-definition-Rk-T-level} and
\begin{equation*}
\begin{split}
   \mathbf{A}_k := & \left(-\beta + 4\alpha L_{\nabla F}^2\right) \| \vd^k \|^2 + \sum_{j=2}^{T}\left(-\gamma_j + 2\alpha (T-1)C_j^2\right) \| \vf_{j}(\vu_{j+1}^k) - \vu_{j}^k \|^2 \\
   & -\gamma_1\| \vf_1(\vu_2^k) - \vu_1^k \|^2-\alpha \| \nabla F(\vx^k) -\vz^k\|^2 + \sum_{j=2}^{T} C_j  \| \vd^k\|\| \vf_j(\vu_{j+1}) - \vu_{j} \|\\
        & + \left(\beta \| \vd^k \|  + \| \nabla F(\vx^k) -\vz^k\| \right) \|\tilde{\vy}^k - \vy^k\|+ \left(4\alpha L_{\nabla F}^2 -\beta\right)\| \tilde{\vy}^k - \vy^k\|^2\\
         &  +2 \left( H_k(\tilde{\vy}^k) - H_k(\vy^k)\right).
\end{split}
\end{equation*}
We can further provide a simplified upper bound for $\mathbf{A}_k$. By Young's inequality, we have
\begin{equation*}
    \begin{split}
        \beta \| \vd^k \| \| \tilde{\vy}^k - \vy^k \| &\leq \frac{\beta}{4} \| \vd^k \|^2 + \beta \|   \tilde{\vy}^k - \vy^k\|^2,\\
        \| \nabla F(\vx^k) - \vz^k \| \|\tilde{\vy}^k - \vy^k \| &\leq \frac{\alpha}{2} \| \nabla F(\vx^k) - \vz^k \|^2 + \frac{1}{2\alpha} \|\tilde{\vy}^k - \vy^k \|^2\\
        C_j  \| \vd^k\|\| \vf_j(\vu_{j+1}) - \vu_{j} \| &\leq \frac{\alpha L_{\nabla F}^2}{ T-1}\| \vd^k\|^2 + \frac{ (T-1) C_j^2}{4\alpha  L_{\nabla F}^2}\| \vf_j(\vu_{j+1}) - \vu_{j} \|^2.
    \end{split}
\end{equation*}
Thus, 
\begin{equation*}
\begin{split}
   \mathbf{A}_k \leq & \left(-\frac{3\beta}{4} + 5\alpha L_{\nabla F}^2\right) \| \vd^k \|^2 -\gamma_1\| \vf_1(\vu_2^k) - \vu_1^k \|^2 -\frac{\alpha}{2}\| \nabla F(\vx^k) -\vz^k\|^2 \\
   &+ \sum_{j=2}^{T}\left(-\gamma_j + \left(2\alpha + \frac{1}{4\alpha L_{\nabla F}^2}\right) (T-1)C_j^2\right) \| \vf_{j}(\vu_{j+1}^k) - \vu_{j}^k \|^2 \\
         & + \left(4\alpha L_{\nabla F}^2 +\frac{1}{2\alpha}\right)\| \tilde{\vy}^k - \vy^k\|^2 +2 \left( H_k(\tilde{\vy}^k) - H_k(\vy^k)\right)\\
\end{split}
\end{equation*}
For any $\beta>0$, let
\begin{equation*}
\alpha = \frac{\beta}{20 L_{\nabla F}^2},\qquad \gamma_1 = \frac{\beta}{2}  ,\qquad \gamma_j =  \left(2\alpha + \frac{1}{4\alpha L_{\nabla F}^2}\right) (T-1)C_j^2 + \frac{\beta}{2},\quad 2\leq j\leq T
\end{equation*}
Then, we have
\begin{equation}\label{eq:lemma-Ak-bound-T-level}
\begin{split}
   \mathbf{A}_k \leq   -\frac{\beta}{2} \left(\| \vd^k \|^2  + \sum_{i=1}^{T} \| \vf_{i}(\vu_{i+1}^k) - \vu_{i}^k \|^2  \right)  -\frac{\beta}{40L_{\nabla F}^2}\| \nabla F(\vx^k) -\vz^k\|^2 \\
          + \left(\frac{12}{5} + \frac{20L_{\nabla F}^2}{\beta^2}\right)\left( H_k(\tilde{\vy}^k) - H_k(\vy^k)\right).\\
\end{split}
\end{equation}
As a result of \eqref{eq:merit-func-diff-T-level} and \eqref{eq:lemma-Ak-bound-T-level}, we can further obtain
\begin{equation*}
\begin{split}
   \tau_k \left( \beta \left[\| \vd^k \|^2  + \sum_{i=1}^{T} \| \vf_{i}(\vu_{i+1}^k) - \vu_{i}^k \|^2  \right]  +\frac{\beta}{20L_{\nabla F}^2}\| \nabla F(\vx^k) -\vz^k\|^2 \right)\\
   \leq 2W_k - 2W_{k+1} + 2\mathbf{R}_k + \tau_k\left(\frac{24}{5} + \frac{40L_{\nabla F}^2}{\beta^2}\right)\left( H_k(\tilde{\vy}^k) - H_k(\vy^k)\right),
\end{split}
\end{equation*}
which immediately implies \eqref{eq:lemma-merit-function-telescoping-T-level} by telescoping. 
\end{proof}

\begin{proposition}\label{prop:Rk}
Let $\mathbf{R}_k$ be defined in \eqref{eq:lemma-definition-Rk-T-level} and $\tau_0=1$. Then, under Assumption \ref{aspt:oracle-T-level}, we have
\begin{equation*}
    \E[\mathbf{R}_k |\setF_k ] \leq \hat{\sigma}^2 \tau_k^2, \quad \forall k\geq 1,
\end{equation*}
where
\begin{equation}\label{eq:definition-sigma-hat}
\begin{split}
    \hat{\sigma}^2 := &\sum_{i=1}^{T} \gamma_i \left(\left[4L_{\nabla \vf_i}^2 + L_{\nabla \vf_i}\sqrt{\sigma_{\vG_i}^2 + (4L_{\vf_i}^2 + \sigma_{\vJ_i}^2)c_{i+1}} + \sigma_{\vJ_i}^2\right]c_{i+1} + \sigma_{\vG_i}^2\right) \\
    &+ (\alpha + 2L_{\eta}) \prod_{i=1}^{T} \hat{\sigma}_{\vJ_i}^2 + \frac{L_{\nabla F} + L_{\eta}}{2} D_{\setX}^2.
\end{split}
\end{equation}
\end{proposition}
\begin{proof}
Note that under Assumption \ref{aspt:oracle-T-level}, we have, for $1\leq i\leq T$,
\begin{equation*}
\begin{split}
    &\E[\Delta^{k+1}|\setF_k] = 0,\quad \E[\dot{r}_i^{k+1}|\setF_k] = 0,\quad \E[\dddot{r}^{k+1}|\setF_k] = 0,\\
    &\E[\| \Delta_{\vG_i}^{k+1}\|^2|\setF_k]\leq \sigma_{\vG_i}^2, \quad \E[\|\Delta_{\vJ_i}^{k+1}\|^2 | \setF_k]\leq \sigma_{\vJ_i}^2,
\end{split}
\end{equation*}
and
\begin{equation*}
    \E [\|\Delta^{k+1} \|^2 |\setF_k] \leq \E\left[\left\|\prod_{i=1}^{T}\vJ^{k+1}_{T-i+1} \right\|^2\middle|\setF_k\right]\leq  \prod_{i=1}^{T}\E\left[\left\|\vJ^{k+1}_{T-i+1} \right\|^2\middle|\setF_k\right] \leq \prod_{i=1}^{T} \hat{\sigma}_{\vJ_i}^2.
\end{equation*}
In addition, by Lemma \ref{lem:basic-ineq-T-level} and H\"{o}lder's inequality. we have $\E[\| \vu_{i}^{k+1} - \vu_{i}^k \|^2|\setF_k]\leq c_i \tau_k^2$ and 
\begin{equation*}
\begin{split}
    &\E[\| \vf_i(\vu_{i+1}^{k+1}) - \vu_{i}^{k+1} \|\|\vu_{i}^{k+1} - \vu_{i}^{k}\|^2|\setF_k]\\
    &\leq \E[\| \vf_i(\vu_{i+1}^{k+1}) - \vu_{i}^{k+1} \||\setF_k] \E[\|\vu_{i}^{k+1} - \vu_{i}^{k}\|^2|\setF_k]\\
    &\leq \left(\E[\| \vf_i(\vu_{i+1}^{k+1}) - \vu_{i}^{k+1} \||\setF_k]\right)^{\frac{1}{2}} \E[\|\vu_{i}^{k+1} -\vu_{i}^{k}\|^2|\setF_k]\\
    &\leq c_i \sqrt{\sigma_{\vG_i}^2 + (4L_{\vf_i}^2 + \sigma_{\vJ_i}^2) c_{i+1}}~\tau_k^2.
\end{split}
\end{equation*}
Lastly, from eq.(28) of Proposition 2.1 in \cite{balasubramanian2020stochastic}, we have for any $k\geq 1$,
\begin{equation*}
    \E[\| \vz^{k+1} - \vz^{k} \|^2 | \setF_k] \leq 4\tau_k^2 \prod_{i=1}^{T} \hat{\sigma}_{\vJ_i}^2.
\end{equation*}
The proof is completed by combing all above observations with the expression of $\mathbf{R}_k$ in \eqref{eq:lemma-definition-Rk-T-level}.
\end{proof}

\begin{proof} [Proof of Theorem \ref{thm:main-T-level}] We now present the proof of Theorem \ref{thm:main-T-level}. Note that by Lemma \ref{lem:merit-function-T-level} and given values of $\alpha, \gamma$ in \eqref{eq:alpha-gamma-values}, we obtain
\begin{equation*}
\begin{split}
   &\sum_{k=1}^{N}\tau_k \left[ \beta \left(\| \vd^k \|^2  + \sum_{i=1}^{T} \| \vf_{i}(\vu_{i+1}^k) - \vu_{i}^k \|^2  \right)  +\frac{\beta}{20 L_{\nabla F}^2}\| \nabla F(\vx^k) -\vz^k\|^2 \right]\\
   &\leq 2W_{\alpha,\gamma}(\vx^0, \vz^0, \vu^0) + 2\sum_{k=0}^{N} \mathbf{R}_k + \left(\frac{24}{5} + \frac{40L_{\nabla F}^2}{\beta^2}\right)\sum_{k=0}^{N}\tau_k\left( H_k(\tilde{\vy}^k) - H_k(\vy^k)\right),
\end{split}
\end{equation*}
Taking expectation of both sides and noting that $\E[\mathbf{R}_k | \setF_k] \leq \hat{\sigma}^2 \tau_k^2$ by Proposition \ref{prop:Rk}, we have
\begin{equation}\label{eq:thm-expection-basic}
    \begin{split}
   &\sum_{k=1}^{N} \tau_k\E\left[ \rho \left(\| \vd^k \|^2  + \sum_{i=1}^{T} \| \vf_{i}(\vu_{i+1}^k) - \vu_{i}^k \|^2  \right)  +\alpha\| \nabla F(\vx^k) -\vz^k\|^2 \middle| \setF_{k-1}\right]\\
   &\leq 2W_{\alpha,\gamma}(\vx^0, \vz^0, \vu^0) + 2\hat{\sigma}^2 \sum_{k=0}^{N} \tau_k^2 + \left(\frac{24}{5} + \frac{40L_{\nabla F}^2}{\beta^2}\right)\sum_{k=0}^{N}\tau_k\left( H_k(\tilde{\vy}^k) - H_k(\vy^k)\right).
\end{split}
\end{equation}
Then, setting $\tau_k, t_k$ to be values in \eqref{thm:parameter-value-T-level} and noting that by Lemma \ref{lem:fw}, we have
\begin{equation*}
    H_k(\tilde{\vy}^k) - H_k(\vy^k)\leq \frac{2\beta D_{\setX}^2 (1+\delta)}{t_k+2} \leq \frac{2\beta D_{\setX}^2 (1+\delta)}{\sqrt{k}}, \quad \forall k\geq 1.
\end{equation*}
Also, with the choice of $\vz^0=0$, we have $\vy^0=\tilde{\vy}^0=\vx^0$. Thus, we can conclude that
\begin{equation*}
    \sum_{k=0}^{N}\tau_k\left( H_k(\tilde{\vy}^k) - H_k(\vy^k)\right) \leq  \frac{2\beta D_{\setX}^2 (1+\delta)}{\sqrt{N}}\sum_{k=1}^{N} \frac{1}{\sqrt{k}} \leq 4\beta D_{\setX}^2 (1+\delta).
\end{equation*}
which together with \eqref{eq:thm-expection-basic} immediately imply that $\forall N \geq 1$,
\begin{equation*}
\begin{split}
   \frac{1}{\sqrt{N}}\sum_{k=1}^{N} \E\left[ \beta \left(\| \vd^k \|^2  + \sum_{j=1}^{T} \| \vf_{j}(\vu_{j+1}^k) - \vu_{j}^k \|^2  \right)  + \frac{\beta}{20 L_{\nabla F}^2}\| \nabla F(\vx^k) -\vz^k\|^2\middle|\setF_{k-1} \right]\\
   \leq 2W_{\alpha,\gamma}(\vx^0, \vz^0, \vu^0) + \mathcal{B}(\beta, \sigma^2, L, D_{\setX}, T, \delta).
\end{split}
\end{equation*}
where 
\begin{equation*}
    \mathcal{B}(\beta, \sigma^2, L, D_{\setX}, T, \delta) = 4\hat{\sigma}^2 + 32\beta D_{\setX}^2(1+\delta) \left(\frac{3}{5} + \frac{5L_{\nabla F}^2}{\beta^2}\right),
\end{equation*}
and $\hat{\sigma}^2$ is given in \eqref{eq:definition-sigma-hat}.
As a result, we can obtain \eqref{eq:thm-gradient-mapping} and \eqref{eq:thm-function-value-error} by the definition of random integer $R$ and
\begin{equation*}
\begin{split}
  \| \mathcal{G}(\vx^k, \nabla F(\vx^k), \beta) \|^2 &\leq  2\beta^2 \| \vd^k \|^2 +2\beta^2\left\| \proj_{\setX}\left(\vx^k- \frac{1}{\beta}\nabla F(\vx^k)\right)-\proj_{\setX}\left(\vx^k- \frac{1}{\beta}\vz^k\right)\right\|^2\\
  &\leq 2\beta^2 \| \vd^k \|^2 + 2\|\nabla F(\vx^k)-\vz^k\|^2.
\end{split}
\end{equation*}
\end{proof}

\section{Proofs for Section~\ref{sec:specialcase}}\label{sec:proof-thm-one-two}

\subsection{Proof of Theorem~\ref{thm:two-level} for $T=2$}

To show the rate of convergence for Algorithm \ref{alg:CG-NASA},  we simplify the merit function in the analysis of the multi-level problems and leverage the following function: 
\begin{equation}\label{eq:definition-merit-fucntion}
    W_{\alpha,\gamma}(\vx^k, \vz^k, \vu^k) = F(\vx^k) -F^\star - \eta (\vx^{k}, \vz^k) + \alpha\| \nabla F(\vx^k) - \vz^k \|^2 + \gamma \| \vf_2(\vx^k) - \vu_2^k\|^2,
\end{equation}
where $\alpha, \gamma$ are positive constants, $\eta(\cdot, \cdot)$ is defined in \eqref{eq:definition-eta}. We now present the analogue of Lemma \ref{lem:merit-function-T-level} for Algorithm \ref{alg:CG-NASA}. The proof follows similar steps as that proof of Lemma \ref{lem:merit-function-T-level} with slight modifications, and hence we will skip some arguments already presented before.

\begin{lemma}
\label{lem:merit-function}
Let $\{\vx^k, \vz^k, \vu_2^k\}_{k\geq 0}$ be the sequence generated by Algorithm \ref{alg:CG-NASA} and the merit function $W_{\alpha, \gamma}(\cdot, \cdot, \cdot)$ be defined in \eqref{eq:definition-merit-fucntion} with
\begin{equation*}
    \alpha = \frac{\rho}{L_{\nabla F}}, \qquad \gamma = 3\rho L_{\nabla f_1}, \qquad \rho > 0.
\end{equation*}
Under Assumptions \ref{aspt:lipschitz-T-level} with $T=2$, setting 
$\beta_k\equiv\beta \geq 6\rho L_{\nabla F} + (2\rho +\frac{2}{3\rho})L_{\nabla f_1}L_{\vf_2}^2$,
we have $\forall N \geq 0$
\begin{equation*}
\begin{split}
     &\rho \sum_{k=0}^{N}\tau_k \bigg(L_{\nabla F} \| \vd^k \|^2 +  L_{\nabla \vf_1}\| \vf_2(\vx^k) - \vu_2^k \|^2 +  \frac{1}{L_{\nabla F}}\| \nabla F(\vx^k) - \vz^k\|^2\bigg) \\
     &\leq 2W_0 +  2\sum_{k=0}^{N}\mathbf{R}_k + \bigg(4+ \frac{2(8\rho + 1/\rho) L_{\nabla F} + 24 \rho L_{\nabla \vf_1} L_{\vf_2}^2 }{\beta}\bigg) \sum_{k=0}^{N} \tau_k \left( H_k(\tilde{\vy}^k) - H_k(\vy^k)\right) 
\end{split}
\end{equation*}
where $\vd^k = \vy^k - \vx^k$, $H_k(\cdot), \vy^k$ are defined in \eqref{eq:definition-etak-yk}, and 
\begin{equation}\label{eq:lemma-2-level-definition}
\begin{split}
    \mathbf{R}_k :=& \tau_k^2 \left[\frac{L_{\nabla F}+ L_{\nabla\eta}}{2} D_{\setX}^2  + \gamma \| \Delta_{\vG_2}^{k+1} \|^2 + \alpha \| \Delta^{k+1} \|^2\right]+ \frac{L_{\eta}}{2} \|\vz^{k+1} -\vz^k \|^2\\
    &\qquad + \tau_k \langle \vd^k , \Delta^{k+1}\rangle + \gamma \dot{r}^{k+1} + \alpha \ddot{r}^{k+1},\\
   \Delta^{k+1}:= &\nabla \vf_2(\vx^{k}) \nabla \vf_1 (\vu_2^k) - \vJ_2^{k+1} \vJ_1^{k+1}, \quad \Delta_{\vG_2}^{k+1}:= \vf_2(\vx^{k}) - \vG_2^{k+1}\\
   \dot{r}^{k+1} :=& 2\tau_k \langle \Delta_{\vG_2}^{k+1}, \vf_2(\vx^{k+1}) - \vf_2(\vx^k) + (1-\tau_k) (\vf_2(\vx^k)-\vu_2^k) \rangle,\\
   \ddot{r}^{k+1} :=& 2\tau_{k} \langle \Delta^{k+1}, (1-\tau_{k}) [ \nabla F(\vx^{k}) - \vz^{k}] + \tau_k[\nabla F(\vx^{k}) - \nabla \vf_2(\vx^{k}) \nabla \vf_1(\vu^k)] \\
   &\qquad + \nabla F(\vx^{k+1}) - \nabla F(\vx^{k}) \rangle.
\end{split}
\end{equation}
\end{lemma}

\begin{proof}[Proof of Lemma~\ref{lem:merit-function}]

1. By the Lipschitzness of $\nabla F$ (Lemma \ref{lem:F-smooth}), we have
\begin{equation}\label{eq:lemma-proof-smoothness-F}
    \begin{split}
        &F (\vx^{k+1}) - F(\vx^{k}) \leq \tau_k\langle \nabla F(\vx^{k}), \vd^k \rangle+ \tau_k\|\nabla F(\vx^{k}) - \vz^{k}\|\| \tilde{\vy}^k - \vy^k \|  \\
        &\qquad + \tau_k \beta \|\vd^k\|\| \tilde{\vy}^k - \vy^k\|+ \tau_k\langle\vz^k+\beta\vd^k, \tilde{\vy}^k - \vy^k\rangle + \frac{L_{\nabla F}\tau_k^2  \|\tilde{\vd}^k\|^2}{2}.\\
    \end{split}
\end{equation}

2. Also, by the Lipschitzness of $\nabla \eta$ (Lemma \ref{lem:eta-smooth}) and the optimality condition of in the definition of $\vy^k$, we have
\begin{equation}\label{eq:lemma-proof-eta-diff}
    \begin{split}
        &\eta (\vx^k, \vz^k) - \eta (\vx^{k+1}, \vz^{k+1}) \leq  -\beta \tau_k \|\vd^k \|^2  + \tau_k \big\langle\vz^k + \beta \vd^k, \tilde{\vy}^k - \vy^k\big\rangle\\ 
        &\qquad - \tau_k\big\langle \vd^k, \nabla \vf_2(\vx^{k}) \nabla f_1(\vu_2^k) \big\rangle +\tau_k \big\langle \vd^k, \Delta^{k+1}\big\rangle+ \frac{L_{\nabla\eta}}{2} \bigg[\tau_k^2 \|\tilde{\vd}^k\|^2 + \| \vz^{k+1} - \vz^k \|^2\bigg].
    \end{split}
\end{equation}

3. In addition, by the Lipschitzness of $\vf_2$ and $\nabla \vf_1$, we have
\begin{equation}\label{eq:lemma-proof-dx-nabla-F-diff}
    \begin{split}
        \langle \vd^k, \nabla F (\vx^{k}) -\nabla \vf_2(\vx^{k}) \nabla f_1(\vu_2^k) \rangle& = \big\langle \vd^k, \nabla \vf_2(\vx^{k})^\top \big[\nabla \vf_1(\vf_2(\vx^k)) -\nabla \vf_1(\vu_2^k)\big] \big\rangle \\
        &\leq  L_{\nabla \vf_1}L_{\vf_2} \| \vd^k \| \| \vf_2(\vx^{k}) -  \vu_2^{k}  \|.
        \end{split}
\end{equation}

4. Moreover, by the update rule, we have
\begin{equation}\label{eq:lemma-proof-function-value-error-diff}
    \begin{split}
        &\| \vf_2(\vx^{k+1}) - \vu_2^{k+1} \|^2 =\| \vf_2(\vx^{k+1}) - \vf_2(\vx^k) + (1-\tau_k) [\vf_2(\vx^k)-\vu_2^k] + \tau_k \Delta_{\vG_2}^{k+1}\|^2\\
        & \qquad =\|(1-\tau_k) [\vf_2(\vx^k)-\vu_2^k] + \vf_2(\vx^{k+1}) - \vf_2(\vx^k)\|^2 + \tau_k^2 \|\Delta_{\vG_2}^{k+1}\|^2 +  \dot{r}^{k+1}\\
        &\qquad \leq (1-\tau_k) \| \vf_2(\vx^k)-\vu_2^k\|^2 + 2\tau_k L_{\vf_2}^2 (\| \vd^k \|^2 + \|\tilde{\vy}^k - \vy^k\|^2) + \tau_k^2 \|\Delta_{\vG_2}^{k+1}\|^2 +  \dot{r}^{k+1}
    \end{split}
\end{equation}        
where $\dot{r}^{k+1} := 2\tau_k \langle \Delta_{\vG_2}^{k+1}, \vf_2(\vx^{k+1}) - \vf_2(\vx^k) + (1-\tau_k) (\vf_2(\vx^k)-\vu_2^k) \rangle$ and the last inequality follows Jensen's inequality for the convex function $\|\cdot\|^2$ as well as
\begin{equation*}
     \left\|\frac{1}{\tau_k}\left[\vf_2(\vx^{k+1}) - \vf_2(\vx^k)\right]\right\|^2 \leq L_{\vf_2}^2 \| \tilde{\vd}^k \|^2\leq 2 L_{\vf_2}^2 (\| \vd^k \|^2 + \|\tilde{\vy}^k - \vy^k\|^2).
\end{equation*}

5. Defining
\begin{equation*}
        \ve^k := \frac{1}{\tau_k} \bigg[\nabla F(\vx^{k+1}) - \nabla F(\vx^{k})\bigg] + \nabla F(\vx^{k}) - \nabla \vf_2(\vx^{k}) \nabla \vf_1(\vu^k),
\end{equation*}
and by the update rule, we have 
\begin{equation}\label{eq:lemma-proof-dual-recursion}
    \begin{split}
        &\| \nabla  F(\vx^{k+1}) - \vz^{k+1}  \|^2 = \| (1-\tau_{k}) [ \nabla F(\vx^{k}) - \vz^{k}] + \tau_{k}[\ve^{k} + \Delta^{k+1}]\|^2\\
        &= \| (1-\tau_{k}) [ \nabla F(\vx^{k}) - \vz^{k}] + \tau_{k}\ve^{k}\|^2 + \tau_{k}^2\|\Delta^{k+1}\|^2+ \ddot{r}^{k+1}\\
        &\leq (1-\tau_{k}) \| \nabla F(\vx^{k}) -\vz^{k}\|^2 + \tau_{k} \| \ve^{k}\|^2+ \tau_{k}^2\|\Delta^{k+1}\|^2+ \ddot{r}^{k+1}\\
    \end{split}
\end{equation}
where $\ddot{r}^{k+1} := 2\tau_{k} \langle \Delta^{k}, (1-\tau_{k}) [ \nabla F(\vx^{k}) - \vz^{k}] + \tau_{k}\ve^{k} \rangle$. We can further upper bound the term $\|\ve^k \|^2$ by
\begin{equation}\label{eq:lemma-proof-e-bound}
    \begin{split}
        \| \ve^k \|^2 &\leq 2 L_{\nabla F}^2  \|\tilde{\vd}^k \|^2 + 2 L_{\nabla \vf_1}^2L_{\vf_2}^2 \|\vf_2(\vx^k) - \vu^k\|^2\\
        &\leq 4 L_{\nabla F}^2  (\|\vd^k \|^2 + \|\tilde{\vy}^k -\vy^k\|^2) + 2 L_{\nabla \vf_1}^2L_{\vf_2}^2 \|\vf_2(\vx^k) - \vu^k\|^2
    \end{split}
\end{equation}

6. By combing \eqref{eq:lemma-proof-smoothness-F}, \eqref{eq:lemma-proof-eta-diff}, \eqref{eq:lemma-proof-dx-nabla-F-diff}, \eqref{eq:lemma-proof-function-value-error-diff}, \eqref{eq:lemma-proof-dual-recursion}, 
\eqref{eq:lemma-proof-e-bound}, rearranging the terms, and noting that $\langle\vz^k+\beta\vd^k, \tilde{\vy}^k - \vy^k\rangle = H_k (\tilde{\vy}^k) - H_k(\vy^k) - (\beta/2)\| \tilde{\vy}^k - \vy^k \|^2$ and $\|\tilde{\vd}^k\|\leq D_{\setX}$, we obtain

\begin{equation}\label{eq:merit-func-diff}
    W_{k+1} -  W_k \leq \tau_k \mathbf{A}_k + \mathbf{R}_k
\end{equation}
where $\mathbf{R}_k$ is defined in \eqref{eq:lemma-2-level-definition} and
\begin{equation*}
\begin{split}
   \mathbf{A}_k := & \left(-\beta + 4\alpha L_{\nabla F}^2 + 2\gamma L_{\vf_2}^2\right) \| \vd^k \|^2   + \left(-\gamma + 2\alpha L_{\nabla \vf_1}^2 L_{\vf_2}^2\right) \| \vf_2(\vx^k) - \vu_2^k \|^2 \\
         &+L_{\nabla \vf_1}L_{\vf_2} \| \vd^k \| \| \vf_2(\vx^{k}) -  \vu_2^{k}  \|  -\alpha \| \nabla F(\vx^k) -\vz^k\|^2 \\
         &+ \left(\beta \| \vd^k \|  + \| \nabla F(\vx^k) -\vz^k\| \right) \|\tilde{\vy}^k - \vy^k\|\\
         & + \left(4\alpha L_{\nabla F}^2 + 2\gamma L_{\vf_2}^2 -\beta\right)\| \tilde{\vy}^k - \vy^k\|^2 +2 \left( H_k(\tilde{\vy}^k) - H_k(\vy^k)\right).
\end{split}
\end{equation*}
We then provide a simplified upper bound for $\mathbf{A}_k$. By the Young's inequality, we have
\begin{equation*}
    \begin{split}
        \beta \| \vd^k \| \| \tilde{\vy}^k - \vy^k \| &\leq \frac{\beta}{4} \| \vd^k \|^2 + \beta \|   \tilde{\vy}^k - \vy^k\|^2,\\
        \| \nabla F(\vx^k) - \vz^k \| \|\tilde{\vy}^k - \vy^k \| &\leq \frac{\alpha}{2} \| \nabla F(\vx^k) - \vz^k \|^2 + \frac{1}{2\alpha} \|\tilde{\vy}^k - \vy^k \|^2.
    \end{split}
\end{equation*}
In addition, we reparametrize $\alpha = \frac{\rho}{L_{\nabla F}}$. Noting that by Lemma \ref{lem:F-smooth} with $T=2$
$$\frac{L_{\nabla \vf_1}^2 L_{\vf_2}^2}{L_{\nabla F}} = \frac{L_{\nabla \vf_1}^2 L_{\vf_2}^2}{L_{\nabla \vf_1} L_{\vf_2}^2 + L_{\vf_1} L_{\nabla \vf_2}} \leq L_{\nabla \vf_1},$$
we therefore have 
\begin{equation*}
    \begin{split}
         \mathbf{A}_k \leq & \left(-\frac{3\beta}{4} + 4\rho L_{\nabla F}  + 2\gamma L_{\vf_2}^2\right) \| \vd^k \|^2   + \left(-\gamma + 2 \rho L_{\nabla \vf_1}\right) \| \vf_2(\vx^k) - \vu_2^k \|^2 \\
         &+L_{\nabla \vf_1}L_{\vf_2} \| \vd^k \| \| \vf_2(\vx^{k}) -  \vu_2^{k}  \| -\frac{\rho}{2L_{\nabla F}}\| \nabla F(\vx^k) -\vz^k\|^2 \\
         &+ \left(4\rho L_{\nabla F} + 2\gamma L_{\vf_2}^2 + \frac{L_{\nabla F}}{2\rho}\right)\| \tilde{\vy}^k - \vy^k\|^2 +2 \left( H_k(\tilde{\vy}^k) - H_k(\vy^k)\right) 
    \end{split}
\end{equation*}
Then, setting $\gamma = 3\rho L_{\nabla f_1}$ and $\beta \geq 6\rho L_{\nabla F} + (2\rho +\frac{2}{3\rho})L_{\nabla f_1}L_{\vf_2}^2$, we can obtain
\begin{equation*}
\begin{split}
     &\left(-\frac{3\beta}{4} + 4\rho L_{\nabla F} + 2\gamma L_{\vf_2}^2\right) \| \vd^k \|^2   + \left(-\gamma + 2 \rho L_{\nabla \vf_1}\right) \| \vf_2(\vx^k) - \vu_2^k \|^2 \\
     &\qquad + L_{\nabla f_1}L_{\vf_2} \|\vd^k \|\| \vf_2(\vx^k) - \vu_2^k \| \leq - \frac{\rho L_{\nabla F}}{2} \| \vd^k \|^2 -  \frac{\rho L_{\nabla f_1}}{2} \| \vf_2(\vx^k) - \vu_2^k\|^2
\end{split}
\end{equation*}
Also, we have $(\beta/2) \| \tilde{\vy}^k - \vy^k\|^2 \leq H_k(\tilde{\vy}^k) - H_k(\vy^k)$. Therefore, we can further bound $\mathbf{A}_k$ by
\begin{equation}\label{eq:lemma-proof-Ak-bound}
    \begin{split}
        \mathbf{A}_k \leq & - \frac{\rho L_{\nabla F}}{2} \| \vd^k \|^2 -  \frac{\rho L_{\nabla f_1}}{2} \| \vg(\vx^k) - \vu^k\|^2 -\frac{\rho}{2 L_{\nabla F}} \| \nabla F(\vx^k) -\vz^k\|^2 \\
        & + \bigg(2+ \frac{(8\rho + 1/\rho) L_{\nabla F} + 12 \rho L_{\nabla \vf_1} L_{\vf_2}^2 }{\beta}\bigg)\left( H_k(\tilde{\vy}^k) - H_k(\vy^k)\right).
    \end{split}
\end{equation}
Telescoping \eqref{eq:merit-func-diff} together with \eqref{eq:lemma-proof-Ak-bound}, we get
\begin{equation*}
\begin{split}
     &\rho \sum_{k=0}^{N}\tau_k \bigg(L_{\nabla F} \| \vd^k \|^2 +  L_{\nabla \vf_1}\| \vf_2(\vx^k) - \vu_2^k \|^2 +  \frac{1}{L_{\nabla F}}\| \nabla F(\vx^k) - \vz^k\|^2\bigg) \\
     &\leq 2W_0 +  2\sum_{k=0}^{N}\mathbf{R}_k + \bigg(4+ \frac{2(8\rho + 1/\rho) L_{\nabla F} + 24 \rho L_{\nabla \vf_1} L_{\vf_2}^2 }{\beta}\bigg) \sum_{k=0}^{N} \tau_k \left( H_k(\tilde{\vy}^k) - H_k(\vy^k)\right) 
\end{split}
\end{equation*}
\end{proof}

\begin{proof}[Proof of Theorem~\ref{thm:two-level}, part (a)] The proof follows the same arguments in the proof of Theorem \ref{thm:main-T-level}. Note that by Lemma \ref{lem:merit-function} and given values of $\alpha, \gamma$ in \eqref{eq:alpha-gamma-values}, we obtain
\begin{equation*}
\begin{split}
   &\rho\sum_{k=1}^{N}\tau_k \left[ L_{\nabla F}\| \vd^k \|^2  + L_{\nabla \vf_1}\| \vf_{2}(\vx^k) - \vu_{2}^k \|^2   +\frac{1}{L_{\nabla F}}\| \nabla F(\vx^k) -\vz^k\|^2 \right]\leq 2W_{\alpha,\gamma}(\vx^0, \vz^0, \vu^0) \\
   &\quad + 2\sum_{k=0}^{N} \mathbf{R}_k + \bigg(4+ \frac{2(8\rho + 1/\rho) L_{\nabla F} + 24 \rho L_{\nabla \vf_1} L_{\vf_2}^2 }{\beta}\bigg)\sum_{k=0}^{N}\tau_k\left( H_k(\tilde{\vy}^k) - H_k(\vy^k)\right).
\end{split}
\end{equation*}
Noting that 
\begin{equation*}
\E[\mathbf{R}_k | \setF_k] =  \tau_k^2 \left[\frac{L_{\nabla F}+ L_{\nabla\eta}}{2} D_{\setX}^2  + \gamma \sigma_{\vG_2}^2 + (\alpha + 2L_{\eta}) \hat{\sigma}_{\vJ_1}^2  \hat{\sigma}_{\vJ_2}^2\right]:=\tau_k^2 \hat{\sigma}^2,
\end{equation*}
and taking expectation of both sides, we can complete the proof with the same arguments in the proof of Theorem \ref{thm:main-T-level}. The constants $\mathcal{C}_1$ and $\mathcal{C}_2$ turn out to be
\begin{equation}\label{eq:definition-C1-C2}
\begin{split}
    &\mathcal{C}_1 = 4\left(\frac{\beta^2}{\rho L_{\nabla F}} + \frac{L_{\nabla F}}{\rho}\right)\bigg\{W_{\alpha, \gamma}(\vx^0, \vz^0, \vu^0) + \hat{\sigma}^2 \\
    &\qquad \qquad \qquad \qquad + 4 D_{\setX}^2 (1+\delta) \left[2\beta + (8\rho + \frac{1}{\rho})L_{\nabla F} + 12\rho L_{\nabla \vf_1} L_{\vf_2}^2)\right]\bigg\},\\
    &\mathcal{C}_2 = \frac{2}{\rho L_{\nabla \vf_1}}\bigg\{W_{\alpha, \gamma}(\vx^0, \vz^0, \vu^0) + \hat{\sigma}^2 \\
    &\qquad \qquad \qquad \qquad + 4 D_{\setX}^2 (1+\delta) \left[2\beta + (8\rho + \frac{1}{\rho})L_{\nabla F} + 12\rho L_{\nabla \vf_1} L_{\vf_2}^2)\right]\bigg\}.
\end{split}
\end{equation}
\end{proof}

\subsection{Proof of Theorem~\ref{thm:two-level} for $T=1$}

To show the rate of convergence for Algorithm \ref{alg:CG-ASA},  we leverage the following merit function: 
\begin{equation}\label{eq:definition-merit-fucntion-1-level}
    W_{\alpha}(\vx^k, \vz^k, \vu^k) = F(\vx^k) -F^\star - \eta (\vx^{k}, \vz^k) + \alpha\| \nabla F(\vx^k) - \vz^k \|^2,
\end{equation}
where $\alpha>0$, $\eta(\cdot, \cdot)$ is defined in \eqref{eq:definition-eta}.

\begin{lemma}
\label{lem:merit-function-1-level}
Let $\{\vx^k, \vz^k\}_{k\geq 0}$ be the sequence generated by Algorithm \ref{alg:CG-ASA} with $\beta_k\equiv \beta>0$ and the merit function $W_{\alpha}(\cdot, \cdot)$ be defined in \eqref{eq:definition-merit-fucntion-1-level} with $\alpha = \frac{\beta}{4L_{\nabla F}^2}$. Under Assumptions \ref{aspt:lipschitz-T-level} with $T=1$, we have $\forall N \geq 0$
\begin{equation*}
\begin{split}
     &\beta \sum_{k=0}^{N}\tau_k \bigg(\| \vd^k \|^2 +  \frac{1}{2L_{\nabla F}^2}\| \nabla F(\vx^k) - \vz^k\|^2\bigg) \\
     &\leq 4W_{\alpha}(\vx^0, \vu^0) +  4\sum_{k=0}^{N}\mathbf{R}_k + \bigg(12+ \frac{16L_{\nabla F}^2 }{\beta^2}\bigg) \sum_{k=0}^{N} \tau_k \left( H_k(\tilde{\vy}^k) - H_k(\vy^k)\right) 
\end{split}
\end{equation*}
where $\vd^k := \vy^k - \vx^k$, $H_k(\cdot), \vy^k$ are defined in \eqref{eq:definition-etak-yk}, $\Delta^{k+1}:= \nabla F (\vx^k) - \vJ_1^{k+1}$, and 
\begin{equation}\label{eq:lemma-1-level-definition}
\begin{split}
    \mathbf{R}_k :=& \tau_k^2 \left[\frac{L_{\nabla F}+ L_{\nabla\eta}}{2} D_{\setX}^2  + \alpha \| \Delta^{k+1} \|^2\right]+ \frac{L_{\eta}}{2} \|\vz^{k+1} -\vz^k \|^2\\
    &\qquad + \tau_k \langle \vd^k , \Delta^{k+1}\rangle + \alpha r^{k+1},\\
   r^{k+1} :=& 2\tau_{k} \langle \Delta^{k+1}, (1-\tau_{k}) [ \nabla F(\vx^{k}) - \vz^{k}] + \nabla F(\vx^{k+1}) - \nabla F(\vx^{k}) \rangle.
\end{split}
\end{equation}
\end{lemma}

\begin{proof} The proof is a essentially a simplified version of the proof of Lemma \ref{lem:merit-function}. Hence, we skip some arguments already presented earlier. 

1. By the Lipschitzness of $\nabla F$, we have
\begin{equation}\label{eq:lemma-proof-smoothness-F-1-level}
    \begin{split}
        &F (\vx^{k+1}) - F(\vx^{k}) \leq \tau_k\langle \nabla F(\vx^{k}), \vd^k \rangle+ \tau_k\|\nabla F(\vx^{k}) - \vz^{k}\|\| \tilde{\vy}^k - \vy^k \|  \\
        &\qquad + \tau_k \beta \|\vd^k\|\| \tilde{\vy}^k - \vy^k\|+ \tau_k\langle\vz^k+\beta\vd^k, \tilde{\vy}^k - \vy^k\rangle + \frac{L_{\nabla F}\tau_k^2  \|\tilde{\vd}^k\|^2}{2}.\\
    \end{split}
\end{equation}

2. Also, by the lipschitzness of $\nabla \eta$ (Lemma \ref{lem:eta-smooth}) and the optimality condition of in the definition of $\vy^k$, we have
\begin{equation}\label{eq:lemma-proof-eta-diff-1-level}
    \begin{split}
        &\eta (\vx^k, \vz^k) - \eta (\vx^{k+1}, \vz^{k+1}) \leq  -\beta \tau_k \|\vd^k \|^2  + \tau_k \big\langle\vz^k + \beta \vd^k, \tilde{\vy}^k - \vy^k\big\rangle\\ 
        &\qquad - \tau_k\big\langle \vd^k, \nabla F(\vx^{k}) \big\rangle +\tau_k \big\langle \vd^k, \Delta^{k+1}\big\rangle+ \frac{L_{\nabla\eta}}{2} \bigg[\tau_k^2 \|\tilde{\vd}^k\|^2 + \| \vz^{k+1} - \vz^k \|^2\bigg].
    \end{split}
\end{equation}

3. By the update rule, we have 
\begin{equation}\label{eq:lemma-proof-dual-recursion-1-level}
    \begin{split}
        &\| \nabla  F(\vx^{k+1}) - \vz^{k+1}  \|^2 = \| (1-\tau_{k}) [ \nabla F(\vx^{k}) - \vz^{k}] +\nabla F(\vx^{k+1}) - \nabla F(\vx^{k})+ \tau_{k} \Delta^{k+1}\|^2\\
        &= \| (1-\tau_{k}) [ \nabla F(\vx^{k}) - \vz^{k}] + \nabla F(\vx^{k+1}) - \nabla F(\vx^{k})\|^2 + \tau_{k}^2\|\Delta^{k+1}\|^2+ r^{k+1}\\
        &\leq (1-\tau_{k}) \| \nabla F(\vx^{k}) -\vz^{k}\|^2 + \frac{1}{\tau_{k}} \| \nabla F(\vx^{k+1}) - \nabla F(\vx^{k})\|^2+ \tau_{k}^2\|\Delta^{k+1}\|^2+ r^{k+1}\\
        &\leq (1-\tau_{k}) \| \nabla F(\vx^{k}) -\vz^{k}\|^2 + \tau_{k} L_{\nabla F}^2\| \tilde{\vd}^k\|^2+ \tau_{k}^2\|\Delta^{k+1}\|^2+ r^{k+1}\\
         &\leq (1-\tau_{k}) \| \nabla F(\vx^{k}) -\vz^{k}\|^2 + 2\tau_{k} L_{\nabla F}^2(\| \vd^k\|^2+\|\tilde{\vy}^k - \vy^k\|^2)+ \tau_{k}^2\|\Delta^{k+1}\|^2+ r^{k+1}\\
    \end{split}
\end{equation}
where $r^{k+1} := 2\tau_{k} \langle \Delta^{k}, (1-\tau_{k}) [ \nabla F(\vx^{k}) - \vz^{k}] + \nabla F(\vx^{k+1}) - \nabla F(\vx^{k}) \rangle$. 

4. By combing \eqref{eq:lemma-proof-smoothness-F-1-level}, \eqref{eq:lemma-proof-eta-diff-1-level} \eqref{eq:lemma-proof-dual-recursion-1-level}, rearranging the terms, and noting that $\langle\vz^k+\beta\vd^k, \tilde{\vy}^k - \vy^k\rangle = H_k (\tilde{\vy}^k) - H_k(\vy^k) - (\beta/2)\| \tilde{\vy}^k - \vy^k \|^2$ and $\|\tilde{\vd}^k\|\leq D_{\setX}$, we obtain

\begin{equation}\label{eq:merit-func-diff-1-level}
    W_{k+1} -  W_k \leq \tau_k \mathbf{A}_k + \mathbf{R}_k
\end{equation}
where $\mathbf{R}_k$ is defined in \eqref{eq:lemma-1-level-definition} and
\begin{equation*}
\begin{split}
   \mathbf{A}_k := & \left(-\beta + 2\alpha L_{\nabla F}^2\right) \| \vd^k \|^2  -\alpha \| \nabla F(\vx^k) -\vz^k\|^2 + \left(\beta \| \vd^k \|  + \| \nabla F(\vx^k) -\vz^k\| \right) \|\tilde{\vy}^k - \vy^k\|\\
         & + \left(2\alpha L_{\nabla F}^2 -\beta\right)\| \tilde{\vy}^k - \vy^k\|^2 +2 \left( H_k(\tilde{\vy}^k) - H_k(\vy^k)\right).
\end{split}
\end{equation*}
We then provide a simplified upper bound for $\mathbf{A}_k$. By the Young's inequality, we have
\begin{equation*}
    \begin{split}
        \beta \| \vd^k \| \| \tilde{\vy}^k - \vy^k \| &\leq \frac{\beta}{4} \| \vd^k \|^2 + \beta \|   \tilde{\vy}^k - \vy^k\|^2,\\
        \| \nabla F(\vx^k) - \vz^k \| \|\tilde{\vy}^k - \vy^k \| &\leq \frac{\alpha}{2} \| \nabla F(\vx^k) - \vz^k \|^2 + \frac{1}{2\alpha} \|\tilde{\vy}^k - \vy^k \|^2.
    \end{split}
\end{equation*}
In addition,  setting $\alpha = \frac{\beta}{4L_{\nabla F}^2}$ and noting $(\beta/2) \| \tilde{\vy}^k - \vy^k\|^2 \leq H_k(\tilde{\vy}^k) - H_k(\vy^k)$, we have 
\begin{equation}\label{eq:lemma-proof-Ak-bound-1level}
    \begin{split}
         \mathbf{A}_k \leq & -\frac{\beta}{4} \| \vd^k \|^2 -\frac{\beta}{8 L_{\nabla F}^2}\| \nabla F(\vx^k) -\vz^k\|^2 +\left(3+\frac{4L_{\nabla F}^2}{\beta^2}\right) \left( H_k(\tilde{\vy}^k) - H_k(\vy^k)\right) \\
    \end{split}
\end{equation}
Telescoping \eqref{eq:merit-func-diff-1-level} together with \eqref{eq:lemma-proof-Ak-bound-1level}, we get
\begin{equation*}
\begin{split}
     &\beta \sum_{k=0}^{N}\tau_k \bigg(\| \vd^k \|^2 +  \frac{1}{2L_{\nabla F}^2}\| \nabla F(\vx^k) - \vz^k\|^2\bigg) \\
     &\leq 4W_{\alpha}(\vx^0, \vu^0) +  4\sum_{k=0}^{N}\mathbf{R}_k + \bigg(12+ \frac{16L_{\nabla F}^2 }{\beta^2}\bigg) \sum_{k=0}^{N} \tau_k \left( H_k(\tilde{\vy}^k) - H_k(\vy^k)\right) 
\end{split}
\end{equation*}
\end{proof}
\begin{proof}[Proof of Theorem~\ref{thm:two-level}, part (b)] Given Lemma~\ref{lem:merit-function-1-level}, the proof follows the same arguments as in the proof of Theorem \ref{thm:main-T-level}. The constant $\mathcal{C}_3$ turns out to be
\begin{equation}\label{eq:definition-C3}
\begin{split}
    \mathcal{C}_3 = 8\left(\beta + \frac{2L_{\nabla F}^2}{\beta}\right)\bigg\{W_{\alpha}(\vx^0, \vu^0) + D_{\setX}^2 \left[(1+\delta)\left(12\beta + \frac{16L_{\nabla F}^2}{\beta}\right) + \frac{L_{\nabla F} + L_{\nabla\eta}}{2}\right]\\
    + \alpha \sigma_{\vJ_1}^2 + 2L_{\eta}\hat{\sigma}_{\vJ_1}^2\bigg\}
\end{split}.
\end{equation}
\end{proof}

\section{High-Probability Convergence for $T=1$}\label{sec:proof-high-prob}

\subsection{Preliminaries}

We provide a short review of sub-gaussian and sub-exponential random variables for completeness.

\begin{definition}\label{def:sub-Gaussian} {\normalfont (Sub-gaussian and Sub-exponential)}
\begin{itemize}[leftmargin=2em]
    \item[(a)] A random variable $X$ is $K$-sub-gaussian if there exists $K>0$ such that $\E [\exp(X^2/K^2)] \leq 2$. 
The sub-gaussian norm of $X$, denoted $\|X\|_{\psi_2}$, is defined to be the smallest $K$. That is to say,
$$\|X\|_{\psi_2} = \inf\left\{ t>0: \E[\exp(X^2/t^2)]\leq 2 \right\}.$$
    \item[(b)] A random variable $X$ is $K$-sub-exponential if there exists $K>0$ such that $\E [\exp(|X|/K)] \leq 2$. 
The sub-exponential norm of $X$, denoted $\|X\|_{\psi_1}$, is defined to be the smallest $K$. That is to say,
$$\|X\|_{\psi_1} = \inf\left\{ t>0: \E[\exp(|X|/t)]\leq 2 \right\}.$$
\end{itemize}
\end{definition}

The above characterization is based on the so-called orlicz norm of a random variable. There are equivalent definitions of sub-gaussian and sub-exponential random variables. We refer readers to Proposition 2.5.2 and Proposition 2.7.1 in \cite{vershynin2018high}. In particular, we will also use another definition of sub-gaussian random variables based on the moment generating function given below.

\begin{lemma}\label{lem:mgf-subgaussian}
{\normalfont (Sub-gaussian M.G.F. \cite{vershynin2018high})}
If a random variable $X$ is $K$-sub-gaussian with $\E[X]=0$, then $\E[\exp(\lambda X)] \leq \exp(c \lambda^2 K^2)~ \forall \lambda \in \reals$, where $cx$ is a absolute constant.
\end{lemma}

In the high probability results we show for the special case with $T=1$, we handle the tail probability for two terms involving the mean-zero noise with sub-gaussian norm, $\|\Delta^{k+1}\|^2$ and $\langle \Delta^{k+1}, \Lambda^{k}\rangle$, where $(\Delta^k)$ and $(\Lambda^k)$ are adapted to $(\setF_k)$. Our proof leverages the following two lemmas to control the probability of these two terms being too large.

\begin{lemma}
 {\normalfont (Sub-exponential is sub-gaussian squared \cite{vershynin2018high})} A random variable $X$
is sub-gaussian if and only if $X^2$
is sub-exponential. Moreover, $\|X^2\|_{\psi_1} = \| X\|_{\psi_2}^2$.
\end{lemma}

\begin{lemma}{\normalfont (Generalized Freedman-type Inequality \cite{harvey2019tight})}\label{lem:freedman-ineq}
Let $(\Omega, \setF, (\setF_i), P)$ be a filtered probability space, $(X_i)$ and $(K_i)$ be adapted to $(\setF_i)$, and $n\in\mathbb{N}$. Suppose for all $i \in [n]$, $K_{i-1}\geq 0$, $\E[X_i|\setF_{i-1}]= 0$, and $\E\left[\exp (\lambda X_i)| \setF_{i-1} \right] \leq \exp(\lambda^2 K_i^2)$.
Then for any $t, b \geq 0, a>0$,
    \begin{equation}
        \Pr\left(\underset{k\in [n]}{\bigcup} \left\{\sum_{i=1}^{k} X_i \geq t \text{ and } 2\sum_{i=1}^{k} K_{i-1}^2 \leq a \sum_{i=1}^{k} X_i + b \right\}\right)  \leq \exp\left(-\frac{t}{4a + 8b/t}\right).
    \end{equation}
\end{lemma}

\subsection{Proof of Theorem \ref{thm:one-level-highprob}}

We start with presenting the lemma below which leverages inequalities in Appendix \ref{sec:proof-high-prob} to show a high-probability upper bound for terms involving in the previous analysis.

\begin{lemma}\label{lem:highprob-results}
Under the conditions of Lemma \ref{lem:merit-function-1-level} and Assumption \ref{aspt:oracle-highprob-1-level}, for any $\delta_1, \delta_2, \delta_3, a >0$, we have
\begin{itemize}[leftmargin=1.6em]
    \item[(a)]with probability at least $1-\delta_1$, 
    $\sum_{k=0}^{N}\tau_k^2\|\Delta^{k+1}\|^2\leq K^2 \log (2/\delta_1) \sum_{k=0}^{N} \tau_k^2;$
    \item[(b)]with probability at least $1-\delta_2$,
    \begin{equation*}
    \sum_{k=0}^{N}\tau_k^2\sum_{i=0}^{k-1} \alpha_{i,k} \| \Delta^{i+1} \|^2 \leq K^2 \log (2/\delta_2) \sum_{k=0}^{N} \tau_k^2,
    \end{equation*}
where $\alpha_{i,k} > 0$ and   $ \sum_{i=0}^{k-1}{\alpha_{i,k}}=1$;
    \item[(c)]with probability at least $1-\delta_3$,
\begin{equation*}
    \begin{split}
    &\sum_{k=0}^{N} \langle \Delta^{k+1}, 2\alpha\tau_k(1-\tau_{k}) [ \nabla F(\vx^{k}) - \vz^{k}] + 2\alpha\tau_k(\nabla F(\vx^{k+1}) - \nabla F(\vx^{k})) + \tau_k \vd^k \rangle \\
    &\leq 4a\log(1/\delta_3) + \frac{\beta^2 K^2}{a L_{\nabla F}^4}\sum_{k=0}^{N} \tau_k^2 (1-\tau_{k}) \left\| \nabla F(\vx^{k}) - \vz^{k}\right\|^2 +  \frac{K^2}{a}\sum_{k=0}^{N} \tau_k^2 (4+\frac{\beta^2\tau_k}{L_{\nabla F}^2})\left\| \vd^k\right\|^2.
    \end{split}
\end{equation*}
\end{itemize}
\end{lemma}

\begin{proof}[Proof of Lemma \ref{lem:highprob-results}]
We first show (a). Using the law of total expectation, we have $\E \left[ \exp \left( \frac{\|\tau_k \Delta^{k+1}\|^2}{\tau_k^2 K^2} \right) \right] \leq 2$, which implies that $\|\tau_k \Delta^{k+1}\|^2$ is $\tau_k^2 K^2$-sub-exponential. Thus, we have with probability at least $1-\delta_1$,
\begin{equation}
    \sum_{k=0}^{N} \tau_k^2 \| \Delta^{k+1} \|^2 \leq K^2 \log (2/\delta_1) \sum_{k=0}^{N} \tau_k^2.
\end{equation}
We then show (b). Let $Z_k = \tau_k^2\left\{\sum_{i=0}^{k-1} \alpha_{i,k} \left\| \Delta^{i+1} \right\|^2 \right\} \forall k\geq 0$. Note that for all $k\geq 0$, $\|\Delta^{k+1}\|^2$ is $K^2$-sub-exponential, which further implies that the sub-exponential norm of $Z_k ~(k>0)$ satisfies $\|Z_k\|_{\psi_1} \leq \tau_k^2 K^2$. Therefore, we have for any $\delta_2>0$, with probability at least $1-\delta_2$,
\begin{equation}
    \sum_{k=0}^{N} Z_k \leq K^2 \log (2/\delta_2) \sum_{k=0}^{N} \tau_k^2.
\end{equation}
To prove (c),  we apply Lemma \ref{lem:mgf-subgaussian} and Lemma \ref{lem:freedman-ineq} with
\begin{equation*}
    \begin{split}
        X_i &= \left\langle \Delta^{k+1}, 2\alpha\tau_k \left\{(1-\tau_{k}) [ \nabla F(\vx^{k}) - \vz^{k}] + \nabla F(\vx^{k+1}) - \nabla F(\vx^{k})\right\} + \tau_k \vd^k \right\rangle, \\
        K_i &= \sqrt{c} K\left\|2\alpha\tau_k\left\{(1-\tau_{k}) [ \nabla F(\vx^{k}) - \vz^{k}] + \nabla F(\vx^{k+1}) - \nabla F(\vx^{k})\right\} + \tau_k \vd^k\right\|, \\
        b &= 0, t = 4a \log (1/\delta_3). 
    \end{split}
\end{equation*}
Noting that $\alpha = \frac{\beta}{4L_{\nabla F}^2}$, we obtain that for all $a >0$ with probability at least $1-\delta_3$, $\sum_{i=0}^{N} X_i \leq 4a\log(1/\delta_3)$ and
\begin{align*}
    \sum_{i=0}^{N} X_i &\leq \frac{2cK^2}{a}\sum_{k=0}^{N} \left\|2\alpha\tau_k\left\{(1-\tau_{k}) [ \nabla F(\vx^{k}) - \vz^{k}] + \nabla F(\vx^{k+1}) - \nabla F(\vx^{k})\right\} + \tau_k \vd^k \right\|^2\\
    &\leq \frac{4cK^2}{a}\sum_{k=0}^{N} \tau_k^2 \left\{4\alpha^2\left\|(1-\tau_{k}) [ \nabla F(\vx^{k}) - \vz^{k}] + \nabla F(\vx^{k+1}) - \nabla F(\vx^{k})\right\|^2 + \left\|\vd^k\right\|^2\right\}\\
    %&\leq \frac{4K^2}{a}\sum_{k=0}^{N} \tau_k^2 \left\{4\alpha^2(1-\tau_{k}) \left\| \nabla F(\vx^{k}) - \vz^{k}\right\|^2 + \frac{4\alpha^2}{\tau_k} \left\| \nabla F(\vx^{k+1}) - \nabla F(\vx^{k})\right\|^2 +  \left\|\vd^k\right\|^2\right\}\\
    &\leq \frac{4cK^2}{a}\sum_{k=0}^{N} \tau_k^2 \left\{4\alpha^2(1-\tau_{k}) \left\| \nabla F(\vx^{k}) - \vz^{k}\right\|^2 + (1+4\alpha^2 L_{\nabla F}^2 \tau_k)\left\| \vd^k\right\|^2\right\}\\
    %&= \frac{4cK^2}{a}\sum_{k=0}^{N} \tau_k^2 \left\{\frac{\beta^2(1-\tau_{k})}{4L_{\nabla F}^4} \left\| \nabla F(\vx^{k}) - \vz^{k}\right\|^2 + (1+\frac{\beta^2\tau_k}{4L_{\nabla F}^2})\left\| \vd^k\right\|^2\right\}\\
    &= \frac{c\beta^2 K^2}{a L_{\nabla F}^4}\sum_{k=0}^{N} \tau_k^2 (1-\tau_{k}) \left\| \nabla F(\vx^{k}) - \vz^{k}\right\|^2 +  \frac{cK^2}{a}\sum_{k=0}^{N} \tau_k^2 (4+\frac{\beta^2\tau_k}{L_{\nabla F}^2})\left\| \vd^k\right\|^2,
\end{align*}
where the third inequality comes from the convexity of $\|\cdot\|^2$ and the Lipschitzness of $\nabla F$.
\end{proof}

Provided with the above lemma and Lemma \ref{lem:merit-function-1-level}, we now present the complete proof of Theorem \ref{thm:one-level-highprob}.

\begin{proof}[Proof of Theorem~\ref{thm:one-level-highprob}]
Given the update rule of $\{\vz^k\}$ and the fact that $\tau_0 =1$, we can obtain
\begin{equation*}
    \vz^k = \sum_{i=0}^{k-1} \alpha_{i,k} \vJ_{1}^{i+1}, \quad \alpha_{i,k} = \frac{\tau_i}{\Gamma_{i+1}}\Gamma_{k}~~ 1\leq i\leq k, \quad \sum_{i=0}^{k-1} \alpha_{i,k} = 1 ~~k\geq 1
\end{equation*}
where $\Gamma_k = \Gamma_1 \prod_{i=1}^{k-1}(1-\tau_i)$ and $\Gamma_1 = 1$. Thus,
\begin{equation*}
    \begin{split}
        \left\| \vz^{k+1} - \vz^k \right\|^2 &= \tau_k^2 \left\| \vJ_1^{k+1} - \vz^k \right\|^2 \leq 2\tau_k^2 \left\{\left\| \vJ_{1}^{k+1} \right\|^2 + \left\| \sum_{i=0}^{k-1} \alpha_{i,k} \vJ_{1}^{i+1} \right\|^2\right\}\\
        &\leq 2\tau_k^2\left\{\left\| \vJ_{1}^{k+1} \right\|^2 +\sum_{i=0}^{k-1} \alpha_{i,k} \left\| \vJ_{1}^{i+1} \right\|^2\right\}\\
        &\leq 4\tau_k^2\left\{\left\| \Delta^{k+1} \right\|^2 + \left\|\nabla F(\vx^{k})\right\|^2 +\sum_{i=0}^{k-1} \alpha_{i,k} \left[\left\| \Delta^{i+1} \right\|^2 + \|\nabla F(\vx^{i})\|^2\right]\right\}\\
        &\leq 4\tau_k^2\left\{\left\| \Delta^{k+1} \right\|^2 +\sum_{i=0}^{k-1} \alpha_{i,k} \left\| \Delta^{i+1} \right\|^2 + 2L_{F}^2 \right\}\\
    \end{split}
\end{equation*}
where the second inequality comes from the convexity of $\|\cdot\|^2$. Therefore, we have
\begin{equation*}
        \sum_{k=0}^{N} \| \vz^{k+1} - \vz^k \|^2 \leq 4\sum_{k=0}^{N} \tau_k^2 \| \Delta^{k+1} \|^2 + 4\sum_{k=0}^{N}\tau_k^2\sum_{i=0}^{k-1} \alpha_{i,k} \| \Delta^{i+1} \|^2 + 8L_F^2 \sum_{k=0}^{N} \tau_k^2
\end{equation*}
Applying Lemma \ref{lem:highprob-results} with $\delta_1 = \delta_2 = \delta_3 = \delta/3$ and $a=\frac{16c \beta K^2}{L_{\nabla F}^2}$ together with Lemma \ref{lem:merit-function-1-level},  we have with probability at least $1-\delta$,
\begin{align*}
    \sum_{k=0}^{N} \mathbf{R}_k =& \sum_{k=0}^{N} \langle \Delta^{k+1}, 2\alpha\tau_k(1-\tau_{k}) [ \nabla F(\vx^{k}) - \vz^{k}] + 2\alpha\tau_k(\nabla F(\vx^{k+1}) - \nabla F(\vx^{k})) + \tau_k \vd^k \rangle\\
    & + (\alpha + 2 L_{\eta})\sum_{k=0}^{N}\tau_k^2 \| \Delta^{k+1} \|^2 + 2L_\eta \sum_{k=0}^{N}\tau_k^2\sum_{i=0}^{k-1} \alpha_{i,k} \| \Delta^{i+1} \|^2\\
    &+ \left[\frac{L_{\nabla F}+ L_{\nabla\eta}}{2} D_{\setX}^2 + 4L_\eta L_F^2\right] \sum_{k=0}^{N}\tau_k^2\\
    \leq & \frac{64\beta K^2}{L_{\nabla F}^2}\log(3/\delta) + \frac{\beta}{16L_{\nabla F}^2}\sum_{k=0}^{N} \tau_k^2 (1-\tau_{k}) \left\| \nabla F(\vx^{k}) - \vz^{k}\right\|^2 +  (\frac{L_{\nabla F}^2}{4\beta}+\frac{\beta}{16})\sum_{k=0}^{N} \tau_k^2\left\| \vd^k\right\|^2\\
    &+  \left[ (\frac{\beta}{4L_{\nabla F}^2}+4L_{\eta}) K^2 \log (6/\delta)+ \frac{L_{\nabla F}+ L_{\nabla\eta}}{2} D_{\setX}^2 + 4L_\eta L_F^2\right] \sum_{k=0}^{N}\tau_k^2
\end{align*}
Thus, noting that $\|\vd^k\|^2 \leq D_{\setX}^2~ \forall k\geq0$, we have with probability at least $1-\delta$, 
\begin{equation*}
\begin{split}
     &\beta \sum_{k=0}^{N}\tau_k \bigg(\| \vd^k \|^2 +  \frac{1}{4L_{\nabla F}^2}\| \nabla F(\vx^k) - \vz^k\|^2\bigg) \\
     &\leq 4W_{\alpha}(\vx^0, \vu^0) +  \bigg(12+ \frac{16L_{\nabla F}^2 }{\beta^2}\bigg) \sum_{k=0}^{N} \tau_k \left( H_k(\tilde{\vy}^k) - H_k(\vy^k)\right)  +\frac{256\beta K^2}{L_{\nabla F}^2}\log(3/\delta) \\
    &+  \left[ \left(\frac{\beta}{L_{\nabla F}^2}+16L_{\eta}\right) K^2 \log (6/\delta)+ \left(\frac{L_{\nabla F}^2}{\beta}+\frac{\beta}{4}+2L_{\nabla F}+ 2L_{\nabla\eta}\right) D_{\setX}^2 + 16L_\eta L_F^2\right] \sum_{k=0}^{N}\tau_k^2
\end{split}
\end{equation*}
Following the same arguments as in the proof of Theorem \ref{thm:main-T-level}, we have with probability at least $1-\delta$,
\begin{equation*}
    \min_{k=1,\dots, N} V(\vx^k, \vz^k) \leq \mathcal{O}\left(\frac{ K^2 \log (1/\delta)}{\sqrt{N}}\right)
\end{equation*}

\end{proof}

\end{document}